\documentclass[a4paper,12pt,reqno]{amsart}

\usepackage{amsfonts}
\usepackage{amsmath}
\usepackage{amssymb}
\usepackage{cite}
\usepackage{mathrsfs}
\usepackage{enumitem}
\usepackage[colorlinks]{hyperref}

\numberwithin{equation}{section}

\usepackage{graphicx}

\newtheorem{theorem}{Theorem}[section]
\newtheorem{defi}[theorem]{Definition}
\newtheorem{remark}[theorem]{Remark}

\begin{document}

\title[Discrete Heat equation with irregular thermal conductivity]
{Discrete Heat Equation  with irregular thermal conductivity and tempered distributional data}
\author[M. Chatzakou]{Marianna Chatzakou}
\address{
	Marianna Chatzakou:
	\endgraf
	Department of Mathematics: Analysis, Logic and Discrete Mathematics
	\endgraf
	Ghent University
	\endgraf
	Ghent
	\endgraf
	Belgium
	\endgraf
	{\it E-mail address} {\rm marianna.chatzakou@ugent.be}
}
\author[A. Dasgupta]{Aparajita Dasgupta}
\address{
	Aparajita Dasgupta:
	\endgraf
	Department of Mathematics
	\endgraf
	Indian Institute of Technology, Delhi, Hauz Khas
	\endgraf
	New Delhi-110016 
	\endgraf
	India
	\endgraf
	{\it E-mail address} {\rm adasgupta@maths.iitd.ac.in}
}

\author[M. Ruzhansky]{Michael Ruzhansky}
\address{
	Michael Ruzhansky:
	\endgraf
	Department of Mathematics: Analysis, Logic and Discrete Mathematics
	\endgraf
	Ghent University, Belgium
	\endgraf
	and
	\endgraf
	School of Mathematical Sciences
	\endgraf
	Queen Mary University of London
	\endgraf
	United Kingdom
	\endgraf
	{\it E-mail address} {\rm ruzhansky@gmail.com}
}
\author[A. Tushir]{Abhilash Tushir}
\address{
	Abhilash Tushir:
	\endgraf
	Department of Mathematics
	\endgraf
	Indian Institute of Technology, Delhi, Hauz Khas
	\endgraf
	New Delhi-110016 
	\endgraf
	India
	\endgraf
	{\it E-mail address} {\rm abhilash2296@gmail.com}
}
\thanks{M. Chatzakou is a postdoctoral fellow of the Research Foundation – Flanders (FWO) under the postdoctoral grant No 12B1223N. A. Dasgupta is supported by Core Research Grant, RP03890G,  Science and Engineering Research Board (SERB), DST,  India. M. Ruzhansky is supported by the EPSRC Grants 
EP/R003025 by the FWO Odysseus 1 grant G.0H94.18N: Analysis and Partial Differential Equations and by the  Methusalem programme of the Ghent University Special Research Fund (BOF) (Grantnumber
01M01021). The last author is supported by the institute assistantship from Indian Institute of Technology Delhi, India. }
\date{\today}
\begin{abstract}
In this paper, we consider a semi-classical version of the nonhomogeneous heat equation with singular time-dependent coefficients on the lattice $\hbar \mathbb{Z}^n$. We establish the well-posedeness of such Cauchy equations in the classical sense when regular coefficients are considered, and analyse how the notion of very weak solution adapts in such equations when distributional coefficients are regarded. We prove the well-posedness of both the classical and the very weak solution in the weighted spaces $\ell^{2}_{s}(\hbar \mathbb{Z}^n)$, $s \in \mathbb{R}$, which is enough to prove the well-posedness in the space of tempered distributions $\mathcal{S}'(\hbar \mathbb{Z}^n)$. Notably, when $s=0$, we show that for $\hbar \rightarrow 0$, the classical (resp. very weak) solution of the heat equation in the Euclidean setting $\mathbb{R}^n$ is recaptured by the  classical (resp. very weak) solution of it in the semi-classical setting $\hbar \mathbb{Z}^n$.
\end{abstract}
\maketitle
\tableofcontents
\section{Introduction}\label{intro}
The solvability of the heat equation, as well as the inverse problem for the heat equation where thermal coefficients are constant, time-dependent, or space-dependent, has been studied extensively in many different settings. 

For the case of the heat equation with singular potentials in the space variable, we refer to the works \cite{marianna_heat2022,MR4207957} for an overview of the known results in different settings. For the case where the time-dependent coefficients are considered, we refer to the works \cite{archil2002,MR3166532,MR4469006,MR3610706,MR4379761} and \cite{MR3320161}, that apply in settings different from the lattice $\mathbb{Z}^n$. In the particular case of the lattice $\mathbb{Z}^n$ we have the work \cite{FNN95} where constant coefficients are considered. In the aforesaid setting, the heat equation with constant coefficients reads as the parabolic Anderson model, see e.g. \cite{MR4076766} and references therein, and the asymptotic analysis of its solution is a topic of wide interest in the field of stochastic analysis.

The setting in this current work is, for a (small) semi-classical parameter $\hbar>0$,  given by
$$\hbar\mathbb{Z}^{n}=\{x\in \mathbb{R}^{n}: x=\hbar k, ~ k\in\mathbb{Z}^n\},$$ and clearly includes the $n$-dimensional lattice $\mathbb{Z}^n$ as a special case. For $\alpha>0$, the discrete fractional Laplacian on $\hbar\mathbb{Z}^{n}$  denoted by $\left(-\mathcal{L}_{\hbar}\right)^{\alpha}$ is defined by
\begin{equation}
    \left(-\mathcal{L}_{\hbar}\right)^{\alpha} u(k):=\sum_{j \in \mathbb{Z}^n} a_j^{(\alpha)} u(k+j \hbar), \quad k\in\hbar\mathbb{Z}^{n},
\end{equation}
where the expansion coefficient $a_{j}^{(\alpha)}$ is given by
\begin{equation}\label{ajformula}
a_j^{(\alpha)}:=\int\limits_{\left[-\frac{1}{2},\frac{1}{2}\right]^{n}}\left[\sum_{l=1}^n 4 \sin ^2\left(\pi \xi_{l}\right)\right]^\alpha e^{-2 \pi i j \cdot \xi} \mathrm{d} \xi.
\end{equation}
For more information on the discrete fractional Laplacian, see Section \ref{sec:pre}.

For the space variable $k \in \hbar \mathbb{Z}^n $ in this setting, we analyse the semi-classical analogue of the heat equation in the Euclidean setting, which reads as follows:
\begin{equation}\label{heatpde}
	\left\{\begin{array}{l}
		\partial_{t}u(t, k)+a(t)\hbar^{-2\alpha}\left(-\mathcal{L}_{\hbar}\right)^{\alpha} u(t, k)+b(t) u(t, k)=f(t,k), \quad (t,k) \in(0, T]\times\hbar\mathbb{Z}^{n}, \\
		u(0, k)=u_{0}(k), \quad k \in \hbar \mathbb{Z}^{n}, 
	\end{array}\right.
\end{equation}
where $a=a(t)\geq 0$ is the thermal conductivity, $b$ is
  a real-valued bounded potential    and $f$  is the heat source. 

 The first aim of the current work is to prove the well-posedness of the classical solution of the heat equation as in \eqref{heatpde} in the semi-classical setting $\hbar \mathbb{Z}^n$. Let us note that in \cite{Wave-JDE,Klein-Arxiv} the authors examine the discrete wave equation with time-dependent coefficients and the discrete Klein-Gordan equation with regular potential and prove that they are well-posed in $\ell^{2}(\hbar\mathbb{Z}^{n})$. In this work, we are extending the well-posedness results in \cite{Wave-JDE,Klein-Arxiv} for our consideration of the Cauchy problem by allowing the initial data  to grow polynomially.  More precisely,  we investigate the well-posedness of the  Cauchy problem \eqref{heatpde} with regular/irregular coefficients as well as the heat source and initial Cauchy data from the space of discrete tempered distributions $\mathcal{S}^{\prime}(\hbar\mathbb{Z}^{n})$.
  A special feature in our approach in this article is that we also allow the coefficients $a,b$ to be distributions. Particularly, our analysis allows $a,b$ to have distributional irregularities; i.e., to have $\delta$-type terms, while also Heaviside discontinuities e.g. we can take $b(t)=\delta+H(t)$. Taking into account that the solution $u(t,x)$ might as well have singularities in $t$, this consideration would lead to foundational mathematical difficulties due to the problem of the impossibility of multiplying distributions, see Schwartz \cite{MR64324}. To overcome this problem we employ the theory of very weak solution as introduced in  \cite{GRweak} which allows us to recapture the classical/distributional solution to the Cauchy problem \eqref{heatpde}, provided that the latter exists.

  With the well-posedness of the classical (resp. very weak) solution to the Cauchy problem \eqref{heatpde} on the lattice $\hbar \mathbb{Z}^n$ at our disposal, two natural questions arise:
  \begin{enumerate}
      \item Can we recapture the classical solution to the heat equation \eqref{heatpde} when the space variable lies in the Euclidean space $\mathbb{R}^n$ by allowing $\hbar \rightarrow 0$?
      \item Can we recapture the very weak solution to \eqref{heatpde} when the space variable lies in the Euclidean space $\mathbb{R}^n$ by allowing $\hbar \rightarrow 0$?
  \end{enumerate}

We will see that both questions are answered in the affirmative, provided additional Sobolev regularity. In particular, we consider the semi-classical limits $\hbar\to 0$ for the following cases:
\begin{enumerate}[label=(\Alph*)]
    \item\label{c1}regular coefficient and Sobolev initial data;
    \item\label{c2} distributional coefficient and Sobolev initial data,
\end{enumerate}
and prove that in both cases we recover the (classical/very weak) solution in the Euclidean setting. The idea of ensuring globally convergence for the solution by adding Sobolev regularity, can be found
in the semi-classical limit theorems in \cite[Theorem 1.2]{Wave-JDE} and \cite[Theorem 1.3]{Klein-Arxiv}.

Note that the Cauchy problem  \eqref{heatpde} is  the discrete analogue of heat equation with time-dependent coefficients in the Euclidean setting $\mathbb{R}^n$ given by: 
\begin{equation}\label{heatpdeEuc}
	\left\{\begin{array}{l}
		\partial_{t}u(t, x)+a(t)(-\mathcal{L})^{\alpha} u(t, x)+b(t) u(t, x)=f(t,x), \quad  (t,x) \in(0, T]\times\mathbb{R}^{n}, \\
		u(0, x)=u_{0}(x), \quad x \in  \mathbb{R}^{n}, 
	\end{array}\right.
\end{equation}
where $(-\mathcal{L})^{\alpha}$ is the usual fractional Laplacian $\left(-\mathcal{L}\right)^{\alpha}$ on $\mathbb{R}^{n}$  defined as a pseudo-differential operator with symbol $|2\pi\xi|^{2\alpha}=\left[\sum\limits_{l=1}^{n}(2\pi \xi_{l})^{2}\right]^{\alpha}$ i.e.,
\begin{equation}\label{frlap}
	(-\mathcal{L})^{\alpha} u(x)=\int\limits_{\mathbb{R}^{n}}\left|2\pi\xi\right|^{2\alpha}\widehat{u}(\xi)e^{2\pi ix\cdot\xi}\mathrm{d}\xi.
\end{equation} 

Theorems \ref{Eucheatwellpo} and \ref{Eucheatwelvws} below state the well-posedness results of the Cauchy problem \eqref{heatpdeEuc} for the above cases in the Euclidean setting $\mathbb{R}^n$.

Before presenting the aforementioned results, let us note that, notation-wise, we write $a\in L_{m}^{\infty}([0,T])$, if $a\in L^{\infty}([0,T])$ is $m$-times weakly differentiable with $\partial^{j}_{t}a\in L^{\infty}([0,T])$, for all $j=1,\dots,m$. Let us recall the usual Sobolev space $H^{m}(\mathbb{R}^{n})$ and its dual $L^{2}_{m}(\mathbb{R}^{n})$ defined as
 \begin{equation*}
     f\in H^{m}(\mathbb{R}^{n})\iff (I-\mathcal{L})^{m/2}f\in L^{2}(\mathbb{R}^{n}),
 \end{equation*}
 and
 \begin{equation}
     g\in L^{2}_{m}(\mathbb{R}^{n})\iff (1+|\xi|^{2})^{m/2}g\in L^{2}(\mathbb{R}^{n}),
 \end{equation}
 respectively.

 In the sequel, we will be writing $A \lesssim B$ whenever there exists a constant $C$, independent of the appearing parameters, such that $A \leq C B$.
 
For the well-posedness of the Cauchy problem \eqref{heatpdeEuc} in the case \ref{c1} we have the following result:
\begin{theorem}\label{Eucheatwellpo}
    Let $m\in \mathbb{R}$ and $f\in L^{2}([0,T];H^{m}(\mathbb{R}^{n}))$. Assume that $a \in L_1^{\infty}([0, T])$ satisfies $\inf\limits_{t \in[0, T]} a(t)=a_0>0$ and $b \in L^{\infty}([0, T])$. If the initial Cauchy data $u_{0} \in H^{m}(\mathbb{R}^{n})$, then the Cauchy problem \eqref{heatpdeEuc} has a unique solution   $u \in C([0,T];H^{m}(\mathbb{R}^{n})) $ which satisfies the estimate
 \begin{equation}
    \|u(t,
    \cdot)\|_{H^{m}(\mathbb{R}^{n})}^{2}\leq C_{T,a,b}\left(\|u_{0}\|^{2}_{H^{m}(\mathbb{R}^{n})}+\|f\|^{2}_{L^{2}([0,T];H^{m}(\mathbb{R}^{n})}\right),
\end{equation}
for all $t\in[0,T]$, where the positive constant $C_{T,a,b}$ is given by
\begin{equation}\label{constantthm}
  C_{T,a,b}=a_{0}^{-1}\|a\|_{L^{\infty}}e^{a_{0}^{-1}(\|a_{t}\|_{L^{\infty}}+2\|a\|_{L^{\infty}}\|b\|_{L^{\infty}}+\|a\|_{L^{\infty}})T}.
\end{equation}
\end{theorem}
Similarly, for the well-posedness of the Cauchy problem \eqref{heatpdeEuc} in the case \ref{c2} we have:
\begin{theorem}\label{Eucheatwelvws}
     Let a and $b$ be distributions with supports included in $[0, T]$ such that $a \geq a_0>0$ for some positive constant $a_0$, and let the source term $f(\cdot, x)$ be a  distribution with support included in $[0, T]$, for all $x\in\mathbb{R}^{n}$. Let $m \in \mathbb{R}$ and $u_{0} \in H^{m}(\mathbb{R}^{n})$. Then the Cauchy problem \eqref{heatpdeEuc} has a unique very weak solution $(u_{\varepsilon})_{\varepsilon}\in L^{2}([0,T];H^{m}(\mathbb{R}^{n}))^{(0,1]}$ of order $m$.
\end{theorem}

For more details about  the very weak solutions for the Cauchy problem \eqref{heatpdeEuc} in the Euclidean setting, we refer to Section \ref{sec:remark}.

To give a synopsis of the topic of very weak solution, we refer to the works \cite{MR3912648,NR:wveqn} on Cauchy problems with singular, time-dependent coefficients in the Euclidean setting. The concept of very weak solution with space-dependent coefficients in the general setting of graded Lie groups has been employed in the works \cite{marianna_schro12022,marianna_schro22022,marianna_heat2022,MR4211066,MR4046183} and \cite{MR4237945} in the Euclidean setting.


 To summarise, we are able to: use the notion of very weak solution, as introduced in \cite{GRweak} in the setting of hyperbolic Cauchy problems with distributional coefficients in space; understand distributionally the Cauchy problem \eqref{heatpde}; and prove its well-posedness.  Precisely, the following facts will be presented in the sequel:
 \begin{itemize}
     \item The Cauchy problem \eqref{heatpde} is well-posed in the weighted spaces $\ell^{2}_{s}(\hbar \mathbb{Z}^n)$, for all $s \in \mathbb{R}$ and admits a very weak solution even when distributional coefficients $a,b$ are considered.
     \item The very weak solution is unique in the sense of Definition \ref{uniquedef}.
     \item If the  coefficients $a,b$ are regular enough so that the Cauchy problem \eqref{heatpde} has a classical solution, then the very weak solution recaptures the classical one. This fact indicates that the notion of very weak solution as adapted to our setting is consistent with the classical one.
     \item We are able to approximate both the classical and the very weak solution to the heat equation as in \eqref{heatpde} in the Euclidean setting $\mathbb{R}^n$, by the corresponding solution in the semi-classical setting $\hbar\mathbb{Z}^n$.
 \end{itemize}
The structure of the paper is as follows: in Section \ref{sec:main} we present our main results. In Section \ref{sec:pre} we discuss the basics of the Fourier analysis in the case of the lattices $\mathbb{Z}^n$ and $\hbar \mathbb{Z}^n$ and of the torus $\mathbb{T}^{n}_{\hbar}$. Additionally to this, the necessary for our analysis functions of spaces and distributions are recalled. In Section \ref{sec:proofs} we provide the proofs of the results on the existence and uniqueness of the very weak solution of the problem we consider, as well as the consistency of it with the classical solution. In Section \ref{sec:limit} we prove the results on the approximation of the (classical/very weak) solution on $\mathbb{R}^n$, by the one on $\hbar \mathbb{Z}^n$. Finally, in Section \ref{sec:remark}, we make some remarks about the very weak solution in the Euclidean setting.

\section{Main results}\label{sec:main}
In this section, we will present our main results for the well-posedness of the Cauchy problem \eqref{heatpde} with regular/irregular coefficients and the discrete tempered distributional initial data. We will also present the semi-classical limit theorems for the classical as well as very weak solution.

 First, we consider the Cauchy problem \eqref{heatpde} with discrete tempered distributional initial Cauchy data  and the regular coefficients $a \in L_1^{\infty}([0, T])$ and $b \in L^{\infty}([0, T])$. In the light of the  relation \eqref{distribution}, proving the well-posedness in the space of weighted spaces $\ell^{2}_{s}(\hbar\mathbb{Z}^{n})$ is sufficient to prove  the well-posedness in the space of discrete tempered distributions $\mathcal{S}^{\prime}(\hbar\mathbb{Z}^{n})$. 
 
 Let us start our exposition of results with the one on the well-posedness theorem (in the classical sense) for the Cauchy problem \eqref{heatpde}:
\begin{theorem}[Classical solution]\label{heatwellpo}
  Let $s \in \mathbb{R}$ and $f \in L^{2}([0, T] ; \ell^{2}_{s}(\hbar\mathbb{Z}^{n}))$. Assume that $a \in L_1^{\infty}([0, T])$ satisfies $\inf\limits_{t \in[0, T]} a(t)=a_0>0$ and $b \in L^{\infty}([0, T])$. If for the initial Cauchy data we have $u_{0} \in \ell^{2}_{s}\left(\hbar\mathbb{\mathbb{Z}}^{n}\right)$, then the Cauchy problem \eqref{heatpde} has a unique solution  $u \in C([0,T];\ell^{2}_{s}(\hbar\mathbb{Z}^{n}))$ satisfying the estimate
 \begin{equation}\label{wellpoest}
    \|u(t,
    \cdot)\|_{\ell^{2}_{s}(\hbar\mathbb{Z}^{n})}^{2}\leq C_{T,a,b}(\|u_{0}\|^{2}_{\ell^{2}_{s}(\hbar\mathbb{Z}^{n})}+\|f\|^{2}_{L^{2}([0,T];\ell^{2}_{s}(\hbar\mathbb{Z}^{n}))}),
\end{equation}
for all $t\in[0,T]$ and $\hbar>0$, where the positive constant $C_{T,a,b}$ is given by
\begin{equation}\label{constantthm}
  C_{T,a,b}=a_{0}^{-1}\|a\|_{L^{\infty}}e^{a_{0}^{-1}(\|a_{t}\|_{L^{\infty}}+2\|a\|_{L^{\infty}}\|b\|_{L^{\infty}}+\|a\|_{L^{\infty}})T}.
\end{equation}
\end{theorem}
Next, we are considering the Cauchy problem \eqref{heatpde} with discrete tempered distributional initial
Cauchy data and, also  allowing the coefficients and the source term to have singularities in the time variable. As we discussed earlier in Section \ref{intro}, in order to deal with such equations where fundamental mathematical difficulties may arise when adapting the classical approach, we have to introduce the notion of a very weak solution. To this end, let us quickly recall the important points for the upcoming analysis: 

Let $a\in\mathcal{D}^{\prime}\left(\mathbb{R}\right)$ be a distribution.
 Using the Friedrichs-mollifier, i.e., a function $\psi \in C_{0}^{\infty}\left(\mathbb{R}\right), \psi \geq 0$ and $\int_{\mathbb{R}} \psi=1$, we are able to construct families of smooth functions $(a_{\varepsilon})_{\varepsilon}$ by regularizing the distributional coefficient $a$ as follows:
\begin{equation}\label{aeps}
	a_{\varepsilon}(t):=\left(a*\psi_{\omega(\varepsilon)}\right)\left(t\right),\quad \psi_{\omega(\varepsilon)}(t)=\dfrac{1}{\omega(\varepsilon)}\psi\left(\dfrac{t}{\omega(\varepsilon)}\right),\quad\varepsilon\in(0,1],
\end{equation}
 where $\omega(\varepsilon)\geq 0$ and $\omega(\varepsilon)\to0$ as $\varepsilon\to 0$. Let us note that since for the purposes of this article, the distributions $a$ and $b$, as in \eqref{heatpde},  are distributions with domain $[0,T]$, it is enough to consider that  $\text{supp}(\psi)\subseteq K$, with  $K=[0,T]$ throughout this article.

 The notions of moderateness and the negligibility for a net of functions/distributions  as follows:
 \begin{defi}\label{cinfm}
	(i)	A net $\left(a_{\varepsilon}\right)_{\varepsilon} \in L^{\infty}_{m}(\mathbb{R})^{(0,1]}$ is said to be $L^{\infty}_{m}$-moderate if for all  $K \Subset \mathbb{R}$, there exist $N \in \mathbb{N}_{0}$ and $c>0$ such that
	\begin{equation*}	\left\|\partial^{k} a_{\varepsilon}\right\|_{L^{\infty}(K)} \leq c \varepsilon^{-N-k},\quad \text{ for all }k=0,1,\dots,m,
	\end{equation*}
	for all $\varepsilon \in(0,1]$.\\
	(ii) 	A net $\left(a_{\varepsilon}\right)_{\varepsilon} \in L^{\infty}_{m}(\mathbb{R})^{(0,1]}$ is said to be $L^{\infty}_{m}$-negligible if for all $K \Subset \mathbb{R}$ and   $q\in\mathbb{N}_{0}$, there exists  $c>0$ such that
	\begin{equation*}
	\left\|\partial^{k} a_{\varepsilon}\right\|_{L^{\infty}(K)} \leq c \varepsilon^{q},\quad \text{ for all }k=0,1,\dots,m,
	\end{equation*}
	for all $\varepsilon \in(0,1]$.\\
		(iii)	A net $\left(u_{\varepsilon}\right)_{\varepsilon} \in L^{2}([0,T];\ell^{2}_{s}(\hbar\mathbb{Z}^{n}))^{(0,1]}$ is said to be $L^{2}([0, T] ;\ell^{2}_{s}(\hbar\mathbb{Z}^{n}))$-moderate if   there exist $N \in \mathbb{N}_{0}$ and $c>0$ such that
	\begin{equation*}
		\|u_{\varepsilon}\|_{L^{2}([0,T];\ell^{2}_{s}(\hbar\mathbb{Z}^{n}))} \leq c \varepsilon^{-N},
	\end{equation*}
	for all $\varepsilon \in(0,1]$.\\
	(iv) A net $\left(u_{\varepsilon}\right)_{\varepsilon} \in L^{2}([0, T] ;\ell^{2}_{s}(\hbar\mathbb{Z}^{n}))^{(0,1]}$ is said to be $L^{2}([0, T] ; \ell^{2}_{s}(\hbar\mathbb{Z}^{n}))$-negligible if for all $q\in\mathbb{N}_{0}$ there exists $c>0$  such that
	\begin{equation*}
\|u_{\varepsilon}\|_{L^{2}([0,T];\ell^{2}_{s}(\hbar\mathbb{Z}^{n}))} \leq c \varepsilon^{q},
	\end{equation*}
	for all  $\varepsilon \in(0,1]$.
\end{defi}

We note that the moderateness requirements are natural in the sense that distributions are moderately regularized. Moreover, by the structure theorems for distributions, we have the following inclusion:
	\begin{equation*}\label{ccsd}
		\text{compactly supported distributions } \mathcal{E}^{\prime}(\mathbb{R}) \subset\left\{L^{2}\text{-moderate families}\right\}.
	\end{equation*}
Therefore, it is possible that the Cauchy problem \eqref{heatpde} may not have a solution in the space of compactly supported distributions $\mathcal{E}^{\prime}(\mathbb{R})$  but it may exist in the space of $L^{2}$-moderate families in some suitable sense.

The notion of a very weak solution to the Cauchy problem \eqref{heatpde} can be viewed as follows:
\begin{defi}\label{vwkdef}
	Let $s \in \mathbb{R}, f \in L^{2}([0, T] ; \ell^{2}_{s}(\hbar\mathbb{Z}^{n}))$, and $u_{0} \in \ell^{2}_{s}(\hbar\mathbb{Z}^{n}) .$ The net $\left(u_{\varepsilon}\right)_{\varepsilon} \in$ $L^{2}([0, T] ; \ell^{2}_{s}(\hbar\mathbb{Z}^{n}))^{(0,1]}$ is a very weak solution of order $s$ of the Cauchy problem \eqref{heatpde} if there exist
	\begin{enumerate}
		\item $L_{1}^{\infty}$-moderate regularisation $a_{\varepsilon}$    of the coefficient $a$;
		\item $L^{\infty}$-moderate regularisation  $b_{\varepsilon}$ of the coefficient  $b$; and
		\item $L^{2}([0, T] ; \ell^{2}_{s}(\hbar\mathbb{Z}^{n}))$-moderate regularisation $f_{\varepsilon}$ of the source term $f$,
	\end{enumerate}
	such that $\left(u_{\varepsilon}\right)_{\varepsilon}$ solves the regularised problem
	\begin{equation}\label{reg}
		\left\{\begin{array}{l}
			\partial_{t} u_{\varepsilon}(t, k)+a_{\varepsilon}(t)\hbar^{-2\alpha}\left(-\mathcal{L}_{\hbar}\right)^{\alpha} u_{\varepsilon}(t, k)+b_{\varepsilon}(t) u_{\varepsilon}(t, k)=f_{\varepsilon}(t, k),\quad t\in(0, T], \\
			u_{\varepsilon}(0, k)=u_{0}(k),\quad k \in \hbar\mathbb{Z}^{n}, 
		\end{array}\right.
	\end{equation}
	for all $\varepsilon \in(0,1]$, and is $L^{2}([0, T] ; \ell^{2}_{s}(\hbar\mathbb{Z}^{n}))$-moderate.
\end{defi}
It should be noted that Theorem \ref{heatwellpo} provides a unique solution to the regularised Cauchy problem \eqref{heatpde} that satisfies estimate \eqref{wellpoest}.

A distribution $a$ is said to be a positive distribution if $\langle a, \psi\rangle \geq 0$ for all  $\psi\in C^{\infty}_{0}(\mathbb{R})$ such that $\psi\geq 0$. Similarly, a distribution $a$ is said to be a strictly positive distribution if there exists a positive  constant $\alpha$ such that $a-\alpha$ is a positive distribution in the previous sense. In other words, $a\geq\alpha>0$ in the support of $a$, where $a\geq \alpha$, means that
\begin{equation}\label{inpdis}
	\langle a-\alpha, \psi\rangle \geq 0, \quad \text{for all } \psi \in C_0^{\infty}(\mathbb{R}),\psi\geq 0.
\end{equation}
Now we can state the existence theorem  for the
Cauchy problem \eqref{heatpde} with distributional coefficients as follows:
\begin{theorem}[Existence]\label{ext}
	Let   $a$ and $b$  be  distributions with  supports contained in $[0, T]$ such that $a\geq a_{0}>0$ for some positive constant $a_{0}$, and  let the source term $f(\cdot,k)$ be a  distribution with  support contained in $[0,T]$, for all $k\in\hbar\mathbb{Z}^{n}$. For $s\in \mathbb{R}$, we assume that the initial Cauchy data $u_0$ satisfies $u_{0}\in \ell^{2}_{s}(\hbar\mathbb{Z}^{n}) $. Then the Cauchy problem \eqref{heatpde} has a very weak solution $\left(u_{\varepsilon}\right)_{\varepsilon} \in$ $L^{2}([0, T] ; \ell^{2}_{s}(\hbar\mathbb{Z}^{n}))^{(0,1]}$  of order s.
\end{theorem}
Next, we define the uniqueness of the very weak solution obtained in Theorem \ref{ext} for the Cauchy problem \eqref{heatpde}. This should be regarded as if the family of very weak solution  is not ``significantly'' affected by the negligible changes in the approximations of the coefficients $a, b$ and of the source term $f$. This can be also regarded as a "stability" property.

Strictly speaking, the notion of the uniqueness of the very weak solution for the Cauchy problem \eqref{heatpde} is formulated as follows:
 \begin{defi}\label{uniquedef}
 	We say that the Cauchy problem \eqref{heatpde} has a $L^{2}([0, T] ; \ell^{2}_{s}(\hbar\mathbb{Z}^{n}))$-unique very weak solution, if
 	\begin{enumerate}
 		\item for all  $L^{\infty}_{1}$-moderate nets $a_{\varepsilon},\tilde{a}_{\varepsilon}$ such that $(a_{\varepsilon}-\tilde{a}_{\varepsilon})_{\varepsilon}$ is $L^{\infty}_{1}$-negligible;
  \item  for all  $L^{\infty}$-moderate nets $b_{\varepsilon},\tilde{b}_{\varepsilon}$ such that $ (b_{\varepsilon}-\tilde{b}_{\varepsilon})_{\varepsilon}$ are $L^{\infty}$-negligible; and
 		\item for all $L^{2}([0, T] ; \ell^{2}_{s}(\hbar\mathbb{Z}^{n}))$-moderate nets $f_{\varepsilon},\tilde{f}_{\varepsilon}$ such that $(f_{\varepsilon}-\tilde{f}_{\varepsilon})_{\varepsilon}$ is \\ $L^{2}([0, T] ; \ell^{2}_{s}(\hbar\mathbb{Z}^{n}))$-negligible,
 	\end{enumerate}
 the net $(u_{\varepsilon}-\tilde{u}_{\varepsilon})_{\varepsilon}$ is $L^{2}([0, T] ; \ell^{2}_{s}(\hbar\mathbb{Z}^{n}))$-negligible, where
 	 $(u_{\varepsilon})_{\varepsilon}$ and $(\tilde{u}_{\varepsilon})_{\varepsilon}$ 	 are the families of
 	solution corresponding to the 	$\varepsilon$-parametrised problems
 		\begin{equation}\label{reg1}
 		\left\{\begin{array}{l}
 			\partial_{t} u_{\varepsilon}(t, k)+a_{\varepsilon}(t)\hbar^{-2\alpha}\left(-\mathcal{L}_{\hbar}\right)^{\alpha}u_{\varepsilon}(t, k)+b_{\varepsilon}(t) u_{\varepsilon}(t, k)=f_{\varepsilon}(t, k),\quad t \in(0, T], \\
 			u_{\varepsilon}(0, k)=u_{0}(k),\quad k \in \hbar\mathbb{Z}^{n}, 
 		\end{array}\right.
 	\end{equation}
 and
 	\begin{equation}\label{reg2}
 	\left\{\begin{array}{l}
 		\partial_{t} \tilde{u}_{\varepsilon}(t, k)+\tilde{a}_{\varepsilon}(t)\hbar^{-2\alpha}\left(-\mathcal{L}_{\hbar}\right)^{\alpha}\tilde{u}_{\varepsilon}(t, k)+\tilde{b}_{\varepsilon}(t) \tilde{u}_{\varepsilon}(t, k)=\tilde{f}_{\varepsilon}(t, k),\quad t \in(0, T], \\
 		\tilde{u}_{\varepsilon}(0, k)=u_{0}(k),\quad k \in \hbar\mathbb{Z}^{n}, 
 	\end{array}\right.
 \end{equation}
respectively.
 \end{defi}
 \begin{remark}
 The Colombeau algebra $\mathcal{G}(\mathbb{R})$ defined as 
\begin{equation*}
	\mathcal{G}(\mathbb{R})=\frac{C^{\infty}\text {-moderate nets }}{C^{\infty}\text {-negligible nets }},
\end{equation*}
 can also be used to formulate the uniqueness for the Cauchy problem \eqref{heatpde}.
   For more details about the Colombeau algebra, we refer to \cite{MR1187755}. The uniqueness of very weak solution in the sense of Colombeau algebra was introduced by Garetto and the third author in \cite{GRweak} and subsequently used in different settings, see e.g. \cite{MRniya1,NR:wveqn}. 
 \end{remark}
 The following theorem gives the
uniqueness of the very weak solution to the Cauchy problem \eqref{heatpde} in the sense of Definition \ref{uniquedef}.
\begin{theorem}[Uniqueness]\label{uniq}
Let   $a$ and $b$  be  distributions with  supports contained in $[0, T]$ such that $a\geq a_{0}>0$ for some positive constant $a_{0}$, and let  the source term $f(\cdot,k)$ be a  distribution with support contained in $[0,T]$, for all $k\in\hbar\mathbb{Z}^{n}$. For  $s \in \mathbb{R}$, let also $u_{0} \in \ell^{2}_{s}(\hbar\mathbb{Z}^{n})$.	 Then the very weak solution of the Cauchy problem \eqref{heatpde} is $L^{2}([0,T];\ell^{2}_{s}(\hbar\mathbb{Z}^{n}))$-unique.
\end{theorem}
In the following theorem, we prove the consistency results with the classical case:
\begin{theorem}[Consistency]\label{cnst}
Let $s\in\mathbb{R}^{n}$ and $f \in L^{2}([0, T], \ell^{2}_{s}(\hbar\mathbb{Z}^{n}))$. Assume that $a \in L_{1}^{\infty}\left([0, T]\right)$ satisfies $\inf\limits_{t \in[0, T]}a(t)= a_{0}>0$  and $b \in L^{\infty}\left([0, T]\right)$. Let also $u_{0} \in \ell^{2}_{s}(\hbar\mathbb{Z}^{n})$ for any $s \in \mathbb{R}$. Then the regularised net $u_\varepsilon$ converges, as $\varepsilon \rightarrow 0$, in $L^{2}([0, T]; \ell^{2}_{s}(\hbar\mathbb{Z}^{n}))$ to the classical solution of the Cauchy problem \eqref{heatpde}.
\end{theorem}
 The following theorem shows that the classical solution of \eqref{heatpde} in the semi-classical setting $\hbar \mathbb{Z}^n$ obtained in Theorem \ref{heatwellpo} for $s=0$, recaptures the classical solution in the Euclidean setting $\mathbb{R}^n$ as in Theorem \ref{Eucheatwellpo}, provided the latter exists. 
\begin{theorem}\label{semlmtclass}
Let $\alpha\in(0,1]$ and $s=0$. Let $u$ and $v$ be the classical solution of the Cauchy problems \eqref{heatpde} on $\hbar \mathbb{Z}^n$ and \eqref{heatpdeEuc} on $\mathbb{R}^n$, respectively. Assume that for the initial Cauchy data we have $u_0 \in H^{m}(\mathbb{R}^{n})$ for some $m\geq 4\alpha$.
 Then  we have the following uniform in $t\in [0,T]$ convergence:
\begin{equation}
\|v(t,\cdot)-u(t,\cdot)\|_{\ell^{2}(\hbar\mathbb{Z}^{n})} \rightarrow 0, \text { as } \hbar \rightarrow 0\,.  
\end{equation}
\end{theorem}
\begin{remark}
    Let us point out that the solution $u,v$ as in Theorem \ref{semlmtclass} are regarded as solution to the Cauchy problems \eqref{heatpde} and \eqref{heatpdeEuc}, respectively,  with the same  Cauchy data $u_0$, source term $f$ and time-dependent coefficients $a,b$. The same hypothesis holds true in Theorem \ref{semlmtvwk} on the convergence of the nets of the very weak solution in the semi-classical and Euclidean settings of the corresponding regularised heat equations.
\end{remark}

The analogous to Theorem \ref{semlmtclass} statement in the ``very weak sense'' reads as follows: 
\begin{theorem}\label{semlmtvwk}
     let $\alpha\in(0,1]$ and $s=0$. Let $\left(u_{\varepsilon}\right)_{\varepsilon}$ and $\left(v_{\varepsilon}\right)_{\varepsilon}$ be the very weak solution of the Cauchy problem \eqref{heatpde} on $\hbar \mathbb{Z}^n$ and   \eqref{heatpdeEuc} on $\mathbb{R}^n$, respectively. Assume that for the initial Cauchy data, we have
     $u_0 \in H^{m}(\mathbb{R}^{n})$ for some $m\geq 4\alpha$.
 Then we have the following uniform in $t\in [0,T]$ convergence
     \begin{equation}
      \left\|v_{\varepsilon}(t,\cdot)-u_{\varepsilon}(t,\cdot)\right\|_{\ell^{2}(\hbar\mathbb{Z}^{n})} \rightarrow 0 \text { as } \hbar \rightarrow 0,
     \end{equation}
 and pointwise for  $\varepsilon \in(0,1]$.
\end{theorem}
\section{Preliminaries}\label{sec:pre}

The Fourier analysis related to the discrete lattice $\mathbb{Z}^{n}$ and the torus $\mathbb{T}^{n}$ has been developed by Turunen and the third author in \cite{ruzhansky+turunen-book}. The pseudo-difference operators and the related symbolic calculus on the weighted sequence space $\ell^{p}_{s}(\mathbb{Z}^{n})$ have been extensively studied in \cite{kibti}. The aim of the section is to recall the preliminaries and important tools related to the discrete lattice $\hbar\mathbb{Z}^{n}$ and the torus $\mathbb{T}^{n}_{\hbar}$ that will be necessary for the analysis that we will follow.

\subsection{Spaces of functions and distributions on the lattice  $\hbar\mathbb{Z}^{n}$}
The Schwartz space $\mathcal{S}(\hbar\mathbb{Z}^{n})$ on the lattice $\hbar\mathbb{Z}^n$ is the space of rapidly decreasing functions $\varphi: \hbar\mathbb{Z}^n \rightarrow \mathbb{C}$; that is, we write $\varphi \in \mathcal{S}(\hbar\mathbb{Z}^{n})$ if for any $M<\infty$ there exists a constant $C_{\varphi, M}$ such that
\begin{equation*}\label{topnorm}
	|\varphi(k)| \leq C_{\varphi, M}(1+|k|)^{-M}, \quad \text { for all } k \in \hbar\mathbb{Z}^{n},    
\end{equation*}
where  $|k|$ stands for the Euclidean  norm of $k \in \hbar\mathbb{Z}^{n}$; 
 i.e., for $k = \hbar(k_{1},\cdots, k_{n})$, we have  $|k|=\hbar \left(\sum\limits_{l=1}^{n} k_{l}^{2}\right)^{\frac{1}{2}}$. The topology on $\mathcal{S}(\hbar\mathbb{Z}^{n})$ is given by the seminorms $p_j$, defined as  
\begin{equation*}
	p_j(\varphi):=\sup\limits_{k \in \hbar\mathbb{Z}^n}(1+|k|)^j|\varphi(k)|,\quad \varphi\in\mathcal{S}(\hbar\mathbb{Z}^{n})\text{ and } j \in \mathbb{N}_{0}.
\end{equation*}
The  topological dual of $\mathcal{S}(\hbar\mathbb{Z}^{n})$ is the space of tempered distributions defined as
\begin{equation*}
    \mathcal{S}^{\prime}(\hbar\mathbb{Z}^{n}):=\mathcal{L}\left(\mathcal{S}(\hbar\mathbb{Z}^{n}), \mathbb{C}\right),
\end{equation*}
 i.e., $ \mathcal{S}^{\prime}(\hbar\mathbb{Z}^{n})$ is the space of all linear continuous functionals on $\mathcal{S}(\hbar\mathbb{Z}^{n})$ of the form
\begin{equation*}
	\varphi \mapsto( u, \varphi):=\sum_{k \in \hbar\mathbb{Z}^n} u(k) \varphi(k),\quad \varphi \in \mathcal{S}(\hbar\mathbb{Z}^{n}).
\end{equation*}
We note that, in contrast to $\mathcal{S}^{\prime}(\mathbb{R}^{n})$, the tempered distributions in $\mathcal{S}^{\prime}(\hbar\mathbb{Z}^{n})$ in the semi-classical setting are pointwise well-defined functions on  $\hbar\mathbb{Z}^{n}$. Furthermore, a tempered distribution $u: \hbar\mathbb{Z}^n \rightarrow \mathbb{C}$ has    polynomial growth at infinity, i.e., there exist positive constants $M$ and $C_{u, M}$ such that
\begin{equation*}
	|u(k)| \leq C_{u, M}(1+|k|)^{M},\quad k\in\hbar\mathbb{Z}^{n}.
\end{equation*}

For $s \in \mathbb{R}$, one can extend the usual space $\ell^2(\hbar\mathbb{Z}^n)$ to the weighted space $\ell^{2}_{s}(\hbar\mathbb{Z}^{n})$ of order $s$ as follows: for  $f: \hbar\mathbb{Z}^n \rightarrow \mathbb{C}$ we write $f \in \ell^{2}_{s}(\hbar\mathbb{Z}^{n})$ whenever the following norm is finite
\begin{equation*}
	\|f\|_{\ell_{s}^{2}(\hbar\mathbb{Z}^{n})}:=\left(\sum_{k \in \hbar\mathbb{Z}^{n}}(1+|k|)^{2s}|f(k)|^{2}\right)^{\frac{1}{2}}\,.
\end{equation*}
Clearly, the weighted $\ell^{2}$-spaces also include $\ell^{2}(\hbar\mathbb{Z}^{n})$ as a special case  when $s=0$.

Observe that the weighted space $\ell^{2}_{s}(\hbar\mathbb{Z}^{n})$ is a Hilbert space when endowed with the natural inner product:
\begin{equation}\label{innerpwl2}
    ( u,v)_{\ell^{2}_{s}(\hbar\mathbb{Z}^{n})}:=\sum_{k\in\hbar\mathbb{Z}^{n}}(1+|k|)^{2s}u(k)\overline{v(k)}\,,
\end{equation}
where $\overline{v(k)}$ stands for the complex conjugate of $v(k)$.

Let us point out that the structure of the  space of tempered distributions $\mathcal{S}^{\prime}(\hbar\mathbb{Z}^{n})$,  as well as this of the Schwartz space $\mathcal{S}(\hbar\mathbb{Z}^{n})$, are closely related to the weighted $\ell^{2}$-spaces. In particular, we have the following useful relations:

\begin{equation}\label{distribution}
\mathcal{S}(\hbar\mathbb{Z}^{n})=\bigcap_{s \in \mathbb{R}} \ell_{s}^{2}(\hbar\mathbb{Z}^{n})\text{ and }	\mathcal{S}^{\prime}(\hbar\mathbb{Z}^{n})=\bigcup_{s \in \mathbb{R}} \ell_{s}^{2}(\hbar\mathbb{Z}^{n}).
\end{equation}
\subsection{Space of periodic functions and distributions on the  torus $\mathbb{T}_{\hbar}^{n}$}

Let us now introduce the torus denoted by $\mathbb{T}_{\hbar}^{n}$ that will be useful for the subsequent analysis, especially when the semi-classical limit $\hbar \rightarrow 0$ is taken.  The torus $\mathbb{T}_{\hbar}^{n}$ can  be realised via the identification   \[\mathbb{T}_{\hbar}^{n}=\left[-\frac{1}{2\hbar},\frac{1}{2\hbar}\right]^{n},\quad \hbar>0.\] Consequently, the space $C^{k}(\mathbb{T}_{\hbar}^{n})$ consists of functions that are $\frac{1}{\hbar}$-periodic and $k$-times continuously differentiable functions on torus $\mathbb{T}_{\hbar}^{n}$. The space $C^{\infty}(\mathbb{T}_{\hbar}^{n})$ of test functions on the torus $\mathbb{T}_{\hbar}^{n}$ can  then be defined  as
\begin{equation*}
    C^{\infty}(\mathbb{T}_{\hbar}^{n}):=\bigcap\limits_{k=1}^{\infty}C^{k}(\mathbb{T}_{\hbar}^{n}).
\end{equation*}
 The Fr\'echet topology on the space of smooth functions $ C^{\infty}(\mathbb{T}_{\hbar}^{n})$ is given by the seminorms $p_{j}$ defined as
 \begin{equation*}
     p_{j}(\psi):=\max\{\|\partial^{\alpha}\psi\|_{C(\mathbb{T}_{\hbar}^{n})}:|\alpha|\leq j\},\quad j\in \mathbb{N}_{0}, \alpha\in \mathbb{N}^{n}_{0}.
 \end{equation*}
The topological dual of  $ C^{\infty}(\mathbb{T}_{\hbar}^{n})$ is the space of periodic distributions defined as
\begin{equation*}
   \mathcal{D}^{\prime}(\mathbb{T}_{\hbar}^{n}):=\mathcal{L}\left(C^{\infty}(\mathbb{T}_{\hbar}^{n}), \mathbb{C}\right),
\end{equation*}
 i.e., it is the  space of all linear continuous functionals on $C^{\infty}(\mathbb{T}_{\hbar}^{n})$ of the form
\begin{equation*}
	\varphi \mapsto \int_{\mathbb{T}_{\hbar}^{n}} \varphi(\xi) \psi(\xi) \mathrm{d} \xi,\quad \psi \in C^{\infty}(\mathbb{T}_{\hbar}^{n}).
\end{equation*}
\subsection{Related semi-classical Fourier analysis} 
 
The Fourier transform operator  \[\mathcal{F}_{\hbar\mathbb{Z}^{n}}:\mathcal{S}(\hbar\mathbb{Z}^{n})\to C^{\infty}(\mathbb{T}_{\hbar}^{n})\]  is defined as
\begin{equation*}
	\mathcal{F}_{\hbar\mathbb{Z}^{n}}u(\xi):=\hbar^{n/2}\sum_{k \in \hbar \mathbb{Z}^{n}} u(k) e^{-2 \pi i k \cdot \xi}, \quad \xi\in\mathbb{T}_{\hbar}^{n}\,.
\end{equation*}
For the inverse Fourier transform operator \[\mathcal{F}^{-1}_{\hbar\mathbb{Z}^{n}}:C^{\infty}(\mathbb{T}_{\hbar}^{n})\to \mathcal{S}(\hbar\mathbb{Z}^{n})\] we have
\begin{equation*}
	\mathcal{F}^{-1}_{\hbar\mathbb{Z}^{n}}v(k):=\hbar^{n/2}\int_{\mathbb{T}_{\hbar}^{n}}v(\xi)e^{2\pi i k\cdot\xi}\mathrm{d}\xi, \quad k\in\hbar\mathbb{Z}^{n}\,,
\end{equation*}
implying that the Fourier inversion formula is given by
\begin{equation}
	u(k)=\hbar^{n/2}\int_{\mathbb{T}_{\hbar}^{n}}\widehat{u}(\xi) e^{2 \pi i k \cdot \xi}  \mathrm{d}\xi, \quad k \in \hbar\mathbb{Z}^{n}.
\end{equation}
The Fourier transform $\mathcal{F}_{\hbar\mathbb{Z}^{n}}$ can be uniquely extended
to the space of tempered distributions $\mathcal{S}^{\prime}(\hbar\mathbb{Z}^{n})$ when realised via the distributional duality 
\begin{equation}
    (\mathcal{F}_{\hbar\mathbb{Z}^{n}}u,\psi):=( u,\iota \circ \mathcal{F}_{\hbar\mathbb{Z}^{n}}^{-1}\psi), 
\end{equation}
where $u\in \mathcal{S}^{\prime}(\hbar\mathbb{Z}^{n})$, $ \psi\in C^{\infty}(\mathbb{T}_{\hbar}^{n}),$ and  $(\iota \circ f)(x)=f(-x)$. 

Hence for the operator $\mathcal{F}_{\hbar\mathbb{Z}^{n}}$ we can write in more generality $\mathcal{F}_{\hbar\mathbb{Z}^{n}}:\mathcal{S}^{\prime}(\hbar\mathbb{Z}^{n})\to\mathcal{D}^{\prime}(\mathbb{T}_{\hbar}^{n})$, and consequently we can define the inverse operator $\mathcal{F}_{\hbar\mathbb{Z}^{n}}^{-1}:\mathcal{D}^{\prime}(\mathbb{T}_{\hbar}^{n})\to\mathcal{S}^{\prime}(\hbar\mathbb{Z}^{n})$. With the use of the latter, we can define the periodic Sobolev spaces $H^{s}(\mathbb{T}_{\hbar}^{n})$, for $s\in\mathbb{R}$ as follows:

\begin{equation*}
	H^{s}(\mathbb{T}_{\hbar}^{n}):=\left\{u\in \mathcal{D}^{\prime}(\mathbb{T}_{\hbar}^{n}):\|u\|_{H^{s}(\mathbb{T}_{\hbar}^{n})}:=\left(\sum_{k \in \hbar\mathbb{Z}^{n}}(1+|k|) ^{2 s}\left|\mathcal{F}^{-1}_{\hbar\mathbb{Z}^{n}}u(k)\right|^{2}\right)^{1/2}<\infty \right\}.
\end{equation*}
For any $s \in \mathbb{R}$ the periodic Sobolev space $H^s\left(\mathbb{T}^{n}_{\hbar}\right)$ is a Hilbert space endowed with the inner product given by
\begin{equation}\label{innerpsh2}
	( u, v)_{H^s\left(\mathbb{T}^{n}_{\hbar}\right)}:=\sum_{k \in \hbar\mathbb{Z}^{n}}(1+|k|)^{2 s} \mathcal{F}^{-1}_{\hbar\mathbb{Z}^{n}}u(k) \overline{\mathcal{F}^{-1}_{\hbar\mathbb{Z}^{n}}v(k)}. 
\end{equation}
Clearly, the periodic Sobolev space also includes $L^{2}(\mathbb{T}_{\hbar}^{n})$ as a special case when $s=0$.

We have the following relation between  the space of test functions $C^{\infty}(\mathbb{T}_{\hbar}^{n})$ and the space of periodic distributions $\mathcal{D}^{\prime}(\mathbb{T}_{\hbar}^{n})$ and the periodic Sobolev spaces $H^s(\mathbb{T}_{\hbar}^{n})$:
\begin{equation*}
	C^{\infty}(\mathbb{T}_{\hbar}^{n})=\bigcap_{s \in \mathbb{R}} H^s(\mathbb{T}_{\hbar}^{n})\text{ and }\mathcal{D}^{\prime}(\mathbb{T}_{\hbar}^{n})=\bigcup_{s \in \mathbb{R}} H^s(\mathbb{T}_{\hbar}^{n}).
\end{equation*}
Combining the inner product \eqref{innerpwl2} and \eqref{innerpsh2} with the Fourier transform, we obtain the following Plancherel formula 
\begin{equation}\label{planch}
	\|f\|_{\ell_{s}^{2}(\hbar\mathbb{Z}^{n})}=\left\|(1+|\cdot|)^{s}f(\cdot)\right\|_{\ell^{2}(\hbar\mathbb{Z}^{n})}=\|\widehat{f}\|_{H^{s}(\mathbb{T}_{\hbar}^{n})},\quad f\in \ell_{s}^{2}(\hbar\mathbb{Z}^{n}).
\end{equation}
\subsection{The discrete fractional Laplacian}
The discrete fractional Laplacian on the  lattice $\hbar\mathbb{Z}^{n}$ can be defined by restricting the usual fractional centered difference operators in the Euclidean setting $\mathbb{R}^{n}$, see \cite[Section 5.4]{duarte}. For more details about the fractional difference operators on $\mathbb{R}^{n}$, we refer to \cite{samko1993fractional}.

Rigorously, for a positive $\alpha>0$ and for $u$ being a complex-valued grid function on $\hbar\mathbb{Z}^{n}$, the discrete fractional Laplacian $\left(-\mathcal{L}_{\hbar}\right)^\alpha$ is defined by
\begin{equation}\label{fracdef}
    \left(-\mathcal{L}_{\hbar}\right)^{\alpha} u(k):=\sum_{j \in \mathbb{Z}^n} a_j^{(\alpha)} u(k+j \hbar), \quad \alpha>0,
\end{equation}
where the generating function $a_j^{(\alpha)}$ is given by
\begin{equation*}\label{ajformula}
a_j^{(\alpha)}:=\int\limits_{\left[-\frac{1}{2},\frac{1}{2}\right]^{n}}\left[\sum_{l=1}^n 4 \sin ^2\left(\pi \xi_l\right)\right]^\alpha e^{-2 \pi i j \cdot \xi} \mathrm{d} \xi.
\end{equation*}
The Fourier transform and the Fourier inversion formula, allow to verify that
\begin{equation}\label{gensum}
\sum_{j \in \mathbb{Z}^n} a_j^{(\alpha)} e^{2 \pi i j \cdot\xi}=\left[\sum_{l=1}^n 4 \sin ^2\left(\pi \xi_l\right)\right]^{\alpha},\quad\xi\in\mathbb{T}_{\hbar}^{n}.
\end{equation}

Further, using the  relation \eqref{gensum} and the shifting property of the Fourier transform, we can compute the Fourier transform of the discrete fractional Laplacian $\left(-\mathcal{L}_{\hbar}\right)^{\alpha}$   as follows:
\begin{eqnarray}\label{lapft}
	\left(\mathcal{F}_{\hbar\mathbb{Z}^{n}}\left(-\mathcal{L}_{\hbar}\right)^{\alpha}u\right)(\xi)&=&\sum_{k\in\hbar\mathbb{Z}^{n}}\left(-\mathcal{L}_{\hbar}\right)^{\alpha}u(k)e^{-2\pi i k\cdot\xi}\nonumber\\
	&=&\sum_{k\in\hbar\mathbb{Z}^{n}}\left(\sum_{j \in \mathbb{Z}^n} a_j^{(\alpha)} u(k+j \hbar)\right)e^{-2\pi i k\cdot\xi}\nonumber\\
	&=&\left(\sum_{j\in\mathbb{Z}^{n}}a_j^{(\alpha)}e^{2\pi i j\hbar\cdot\xi}\right)\widehat{u}(\xi)\nonumber\\
	&=&\left[\sum_{l=1}^n 4 \sin ^2\left(\pi \hbar\xi_{l}\right)\right]^{\alpha}\widehat{u}(\xi),
\end{eqnarray}
for all $\xi\in\mathbb{T}_{\hbar}^{n},$ and consequently the Fourier transform of fractional Laplacian is
\begin{eqnarray}\label{lapeucft}
    \left(\mathcal{F}_{\hbar\mathbb{Z}^{n}}\left(-\mathcal{L}\right)^{\alpha}u\right)(\xi)&=&\left(\left(-\mathcal{L}\right)^{\alpha}u,e^{2\pi i k\cdot\xi}\right)\nonumber\\&=&\left(u,\left(-\mathcal{L}\right)^{\alpha}e^{2\pi i k\cdot\xi}\right)\nonumber\\
    &=&\left(u,|2\pi\xi|^{2\alpha}e^{2\pi i k\cdot\xi}\right)\nonumber\\
    &=&|2\pi\xi|^{2\alpha}\widehat{u}(\xi),
\end{eqnarray}
 for all $\xi\in\mathbb{T}_{\hbar}^{n}$. For more details about the construction and other properties of discrete fractional Laplacian on  $\hbar\mathbb{Z}^{n}$, see  \cite{Klein-Arxiv}.

The consistency of the formula \ref{fracdef} for the discrete fractional Laplacian $(-\mathcal{L}_{\hbar})^{\alpha}$ with the usual discrete Laplacian introduced in the monograph \cite{Wave-JDE}, can be understood by explicitly computing the expansion coefficients $a_{j}^{(\alpha)}$, for $\alpha=1$ and $j\in\mathbb{Z}^{n}$. Indeed, it is easy to check that
\begin{equation*}
    a_{0}^{(1)}=2n,\quad a_{\pm v_{l}}^{(1)}=-1,\text{ and } a_{j}^{(1)}=0,\quad\text{for all } j\neq 0,\pm v_{l},
\end{equation*}
where  $v_{l}$  is the  $l^{t h}$  basis vector in  $\mathbb{Z}^{n}$, having all zeros except for  $1$  at the  $l^{t h}$  component.
This gives
\begin{eqnarray*}
	(-\mathcal{L}_{\hbar})^{1}u(k)&=&2nu(k)-\sum\limits_{j=\pm v_{l}}u(k+j\hbar)\nonumber\\
	&=&2nu(k)-\sum\limits_{l=1}^{n}\left(u(k+v_{l}\hbar)+u(k-v_{l}\hbar)\right),
\end{eqnarray*}
which is usual discrete Laplacian on $\hbar\mathbb{Z}^{n}$. 
\section{Proof of the main results}\label{sec:proofs}

In this section, we give the proofs of all the results presented above in Section \ref{sec:main}. 
Before moving on to proving our first result, let us note that the proof of Theorem \ref{Eucheatwellpo} in the Euclidean setting $\mathbb{R}^n$ follows the same arguments as the ones that are used in the proof of Theorem \ref{heatwellpo} in the setting $\hbar \mathbb{Z}^n$ except for using the inner product of $L^{2}_{m}(\mathbb{R}^{n})$ instead of $H^{s}(\mathbb{T}_{\hbar}^{n})$. Therefore we just prove Theorem \ref{heatwellpo} and the proof of Theorem \ref{Eucheatwellpo} should be considered verbatim.

 \begin{proof}[Proof of Theorem \ref{heatwellpo}]
     Taking the Fourier transform of the Cauchy problem \eqref{heatpde} with respect to $k\in\hbar\mathbb{Z}^{n}$ and using \eqref{lapft}, we obtain
     \begin{equation}\label{transpde}
	\left\{\begin{array}{l}
		\partial_{t}\widehat{u}(t, \xi)+a(t)\nu^{2}(\xi) \widehat{u}(t, \xi)+b(t) \widehat{u}(t, \xi)=\widehat{f}(t,\xi), \quad \text { with } (t,\xi) \in[0, T]\times\mathbb{T}_{\hbar}^{n}, \\
		\widehat{u}(0, \xi)=\widehat{u}_{0}(\xi), \quad \xi \in  \mathbb{T}_{\hbar}^{n},
	\end{array}\right.
\end{equation}
where
\begin{equation}\label{nuvalue}
\nu^{2}(\xi)=\hbar^{-2\alpha}\left[\sum_{l=1}^n 4 \sin ^2\left(\pi\hbar \xi_{l}\right)\right]^{\alpha} \geq 0.
\end{equation}
Define the energy functional for the Cauchy problem \eqref{transpde} by
\begin{equation}\label{energy}
    E(t,\xi):=( a(t)\widehat{u}(t,\xi),\widehat{u}(t,\xi)),\quad (t,\xi) \in[0, T]\times\mathbb{T}_{\hbar}^{n},
\end{equation}
where $(\cdot,\cdot)$ is the inner product in the Sobolev space $H^{s}(\mathbb{T}_{\hbar}^{n})$ given by \eqref{innerpsh2}. It is easy to check that
\begin{equation}\label{engest}
    \inf _{t \in[0, T]} \{a(t)\}(\widehat{u}(t,\xi),\widehat{u}(t,\xi))\leq( a(t)\widehat{u}(t,\xi),\widehat{u}(t,\xi))\leq \sup_{t \in[0, T]} \{a(t)\}(\widehat{u}(t,\xi),\widehat{u}(t,\xi)).
\end{equation}
Since $a \in L_1^{\infty}([0, T])$, there exist two positive constants $a_0$ and $a_1$ such that
\begin{equation}\label{abound}
\inf _{t \in[0, T]} \{a(t)\}=a_0 \quad \text { and } \sup _{t \in[0, T]} \{a(t)\}=a_1 .   
\end{equation}
Combining the  equations \eqref{energy} and \eqref{abound} together with the inequality \eqref{engest}, we obtain the following bounds for the energy functional
\begin{equation}\label{engyest}
    a_{0}\|\widehat{u}(t,\cdot)\|_{H^{s}}^{2}\leq E(t,\xi)\leq     a_{1}\|\widehat{u}(t,\cdot)\|^{2}_{H^{s}},\quad (t,\xi) \in[0, T]\times\mathbb{T}_{\hbar}^{n}.
\end{equation}
 Differentiating the energy functional $E(t,\xi)$ and using  \eqref{transpde}, we obtain
 \begin{eqnarray*}\label{engder}
     \partial_{t}E(t,\xi)&=&( a_{t}(t)\widehat{u}(t,\xi),\widehat{u}(t,\xi))+( a(t)\widehat{u}_{t}(t,\xi),\widehat{u}(t,\xi))+( a(t)\widehat{u}(t,\xi),\widehat{u}_{t}(t,\xi))\nonumber\\
     &=&( a_{t}(t)\widehat{u}(t,\xi),\widehat{u}(t,\xi))-( a^{2}(t)\nu^{2}(\xi)\widehat{u}(t,\xi),\widehat{u}(t,\xi))-\nonumber\\
     &&( a(t)b(t)\widehat{u}(t,\xi),\widehat{u}(t,\xi))+( a(t)\widehat{f}(t,\xi),\widehat{u}(t,\xi))-\nonumber\\
     &&( a(t)\widehat{u}(t,\xi),a(t)\nu^{2}(\xi)\widehat{u}(t,\xi))-( a(t)\widehat{u}(t,\xi),b(t)\widehat{u}(t,\xi))+\nonumber\\
     &&( a(t)\widehat{u}(t,\xi),\widehat{f}(t,\xi))\nonumber\\
          &=&a_{t}(t)( \widehat{u}(t,\xi),\widehat{u}(t,\xi))-2a^{2}(t)( \nu(\xi)\widehat{u}(t,\xi),\nu(\xi)\widehat{u}(t,\xi))-\nonumber\\
     &&2a(t)b(t)( \widehat{u}(t,\xi),\widehat{u}(t,\xi))+a(t)( \widehat{f}(t,\xi),\widehat{u}(t,\xi))+\nonumber\\
     &&a(t)( \widehat{u}(t,\xi),\widehat{f}(t,\xi))\nonumber\\
     &\leq&\left(|a_{t}(t)|+2|a(t)||b(t)|\right) \|\widehat{u}(t,\cdot)\|^{2}_{H^{s}}+2|a(t)||\operatorname{Re}( \widehat{f}(t,\xi),\widehat{u}(t,\xi))|.
 \end{eqnarray*}
 Using the hypothesis that $a\in L_{1}^{\infty}([0,T])$ and $ b\in L^{\infty}([0,T])$, we obtain
 \begin{eqnarray}\label{eest}
     \partial_{t}E(t,\xi)&\leq& \left(\|a_{t}\|_{L^{\infty}}+2\|a\|_{L^{\infty}}\|b\|_{L^{\infty}}\right)\|\widehat{u}(t,\cdot)\|^{2}_{H^{s}}+2\|a\|_{L^{\infty}}\|\widehat{f}(t,\cdot)\|_{H^{s}}\|\widehat{u}(t,\cdot)\|_{H^{s}}\nonumber\\
     &\leq& \left(\|a_{t}\|_{L^{\infty}}+2\|a\|_{L^{\infty}}\|b\|_{L^{\infty}}+\|a\|_{L^{\infty}}\right)\|\widehat{u}(t,\cdot)\|^{2}_{H^{s}}+\|a\|_{L^{\infty}}\|\widehat{f}(t,\cdot)\|^{2}_{H^{s}}.\nonumber\\
 \end{eqnarray}
If we set   $\kappa_{1}=\|a_{t}\|_{L^{\infty}}+2\|a\|_{L^{\infty}}\|b\|_{L^{\infty}}+\|a\|_{L^{\infty}}$ and $\kappa_{2}=\|a\|_{L^{\infty}}$, then putting together  \eqref{engyest} and \eqref{eest}, we obtain
\begin{equation}\label{finalenrgy}
\partial_{t}E(t,\xi)\leq a_{0}^{-1}\kappa_{1}E(t,\xi)+\kappa_{2}\|\widehat{f}(t,\cdot)\|^{2}_{H^{s}}.   
\end{equation}
Applying the Gronwall’s lemma to the inequality \eqref{finalenrgy}, we get
\begin{equation}\label{Eoest}
    E(t, \xi) \leq e^{\int_{0}^{t}a_{0}^{-1}\kappa_{1}\mathrm{d}\tau}\left(E(0, \xi)+\int_{0}^{t}\kappa_{2}\|\widehat{f}(\tau, \cdot)\|^2_{H^{s}}\mathrm{d}\tau\right),
\end{equation}
for all $(t,\xi)\in[0,T]\times\mathbb{T}_{\hbar}^{n}$. Again combining the estimates  \eqref{engyest} and \eqref{Eoest}, we obtain
\begin{eqnarray*}
    a_{0}\|\widehat{u}(t,\cdot)\|^{2}_{H^{s}}\leq E(t,\xi)&\leq &e^{\int_{0}^{t}a_{0}^{-1}\kappa_{1}\mathrm{d}\tau}\left(E(0, \xi)+\int_{0}^{t}\kappa_{2}\|\widehat{f}(\tau, \cdot)\|^2_{H^{s}}\mathrm{d}\tau\right)\nonumber\\
    &\leq&e^{a_{0}^{-1}\kappa_{1}T}\left(a_{1}\|\widehat{u}(0,\cdot)\|^{2}_{H^{s}}+\kappa_{2} \int_{0}^{T}\|\widehat{f}(\tau, \cdot)\|^2_{H^{s}}\mathrm{d}\tau\right).
\end{eqnarray*}
Further, using the Plancherel formula \eqref{planch}, we obtain the required estimate
\begin{equation*}\label{uestclass}
    \|u(t,\cdot)\|^{2}_{\ell^{2}_{s}(\hbar\mathbb{Z}^{n})}\leq C_{T,a,b}\left(\|u_{0}\|^{2}_{\ell^{2}_{s}(\hbar\mathbb{Z}^{n})}+\|f\|^{2}_{L^{2}([0,T];\ell^{2}_{s}(\hbar\mathbb{Z}^{n}))}\right),\quad\text{ for all }t\in[0,T], 
\end{equation*}
where the constant $C_{T,a,b}$ is given by
\begin{equation*}
  C_{T,a,b}=a_{0}^{-1}\|a\|_{L^{\infty}}e^{a_{0}^{-1}(\|a_{t}\|_{L^{\infty}}+2\|a\|_{L^{\infty}}\|b\|_{L^{\infty}}+\|a\|_{L^{\infty}})T}.
\end{equation*}
 The uniqueness of the solution follows immediately from the above estimate.  This completes the proof.
 \end{proof}
 Thus, we have obtained the well-posedness for the Cauchy problem \eqref{heatpde} in the weighted spaces $\ell^{2}_{s}(\hbar\mathbb{Z}^{n}))$ for all $s\in\mathbb{R}$ and consequently we have distributional well-posedness for the Cauchy problem \eqref{heatpde} in the space of tempered distribution $\mathcal{S}^{\prime}(\hbar\mathbb{Z}^{n})$. We will now prove the existence of the very weak solution to the Cauchy problem \eqref{heatpde} in the case of distributional coefficients. 
 \begin{proof}[Proof of Theorem \ref{ext}]
 Consider the regularised Cauchy problem
		\begin{equation}\label{vwprf}
		\left\{\begin{array}{l}
			\partial_{t} u_{\varepsilon}(t, k)+a_{\varepsilon}(t)\hbar^{-2\alpha}\left(-\mathcal{L}_{\hbar}\right)^{\alpha}u_{\varepsilon}(t, k)+b_{\varepsilon}(t) u_{\varepsilon}(t, k)=f_{\varepsilon}(t, k), \quad t \in(0, T],\\
			u_{\varepsilon}(0, k)=u_{0}(k),\quad k \in \hbar\mathbb{Z}^{n},
		\end{array}\right.
	\end{equation} where the coefficient $a_{\varepsilon}$ is $L^{\infty}_{1}$-moderate, $b_{\varepsilon}$ is $L^{\infty}$-moderate and the source term  $f_{\varepsilon}$  is $L^{2}([0,T];\ell^{2}_{s}(\hbar\mathbb{Z}^{n})$-moderate regularisation of the coefficients $a,b$ and the source term $f$, respectively.
 Taking the Fourier transform  with respect to $k\in\hbar\mathbb{Z}^{n}$, we obtain 
 \begin{equation*}\label{regtranspde}
	\left\{\begin{array}{l}
			\partial_{t} \widehat{u}_{\varepsilon}(t, \xi)+a_{\varepsilon}(t)\nu^{2}(\xi)\widehat{u}_{\varepsilon}(t, \xi)+b_{\varepsilon}(t) \widehat{u}_{\varepsilon}(t, \xi)=\widehat{f}_{\varepsilon}(t, \xi),\quad(t, \xi) \in(0, T] \times \mathbb{T}_{\hbar}^{n}, \\
			\widehat{u}_{\varepsilon}(0, \xi)=\widehat{u}_{0}(\xi),\quad k \in \mathbb{T}_{\hbar}^{n},
		\end{array}\right.
\end{equation*}
where $\nu^{2}(\xi)$ is given by \eqref{nuvalue}. Define the energy functional for the Cauchy problem \eqref{vwprf} by
\begin{equation}\label{weakenergy}
    E_{\varepsilon}(t,\xi):=( a_{\varepsilon}(t)\widehat{u}_{\varepsilon}(t,\xi),\widehat{u}_{\varepsilon}(t,\xi)),\quad (t,\xi) \in[0, T]\times\mathbb{T}_{\hbar}^{n},
\end{equation}
where $(\cdot,\cdot)$ is the inner product in the Sobolev space $H^{s}(\mathbb{T}_{\hbar}^{n})$.
 It is easy to check that
\begin{equation*}\label{weakengest}
    \inf _{t \in[0, T]} \{a_{\varepsilon}(t)\}(\widehat{u}_{\varepsilon}(t,\xi),\widehat{u}_{\varepsilon}(t,\xi))\leq( a_{\varepsilon}(t)\widehat{u}_{\varepsilon}(t,\xi),\widehat{u}_{\varepsilon}(t,\xi))\leq \sup_{t \in[0, T]} \{a_{\varepsilon}(t)\}(\widehat{u}_{\varepsilon}(t,\xi),\widehat{u}_{\varepsilon}(t,\xi)).
\end{equation*}
Since $a$ and $b$ are  distributions, by the structure theorem for compactly supported distributions, there exist $L_{1}, L_{2} \in \mathbb{N}$ and $c_{1}, c_{2}>0$ such that
\begin{equation}\label{std}
	\left|\partial_{t}^{k} a_{\varepsilon}(t)\right| \leq c_{1} \omega(\varepsilon)^{-L_{1}-k}\text{  and  }\left|\partial_{t}^{k} b_{\varepsilon}(t)\right| \leq c_{2} \omega(\varepsilon)^{-L_{2}-k},\quad k \in \mathbb{N}_{0},
\end{equation}
for all  $t \in[0, T]$, where $\omega(\varepsilon)$ is given by \eqref{aeps}. Since $a\geq a_{0}>0$, we can write
\begin{equation}\label{amin}
	a_{\varepsilon}(t)=\left(a*\psi_{\omega(\varepsilon)}\right)(t)=\langle a,\tau_{t}\tilde{\psi}_{\omega(\varepsilon)}\rangle\geq\tilde{a}_{0}>0,
\end{equation}
where $\tilde{\psi}(x)=\psi(-x),x\in\mathbb{R}$ and $\tau_{t}\psi(\xi)=\psi(\xi-t),\xi\in\mathbb{R}$.

 Now applying  Theorem \ref{heatwellpo} to the Cauchy problem \eqref{vwprf}, we have the following estimate
\begin{equation}\label{uepsinorm}
    \|u_{\varepsilon}(t,\cdot)\|^{2}_{\ell^{2}_{s}(\hbar\mathbb{Z}^{n})}\leq C_{T,a_{\varepsilon},b_{\varepsilon}}\left(\|u_{0}\|^{2}_{\ell^{2}_{s}(\hbar\mathbb{Z}^{n})}+\|f_{\varepsilon}\|^{2}_{L^{2}([0,T];\ell^{2}_{s}(\hbar\mathbb{Z}^{n}))}\right),\quad\text{ for all }t\in[0,T], 
\end{equation}
where the constant $C_{T,a_{\varepsilon},b_{\varepsilon}}$ is given by
\begin{equation}\label{constcte}
  C_{T,a_{\varepsilon},b_{\varepsilon}}=\tilde{a}_{0}^{-1}\|a_{\varepsilon}\|_{L^{\infty}}e^{\tilde{a}_{0}^{-1}(\|\partial_{t}a_{\varepsilon}\|_{L^{\infty}}+2\|a_{\varepsilon}\|_{L^{\infty}}\|b_{\varepsilon}\|_{L^{\infty}}+\|a_{\varepsilon}\|_{L^{\infty}})T}.
\end{equation}
Combining the estimates \eqref{std} and \eqref{uepsinorm} with \eqref{constcte} we get
\begin{eqnarray*}\label{constnorm}
    \|u_{\varepsilon}(t,\cdot)\|^{2}_{\ell^{2}_{s}(\hbar\mathbb{Z}^{n})}&\leq&\tilde{a}_{0}^{-1}c_{1}\omega(\varepsilon)^{-L_{1}}e^{\left(c_{1}\omega(\varepsilon)^{-L_{1}-1}+2c_{1}c_{2}\omega(\varepsilon)^{-L_{1}-L_{2}}+c_{1}\omega(\varepsilon)^{-L_{1}}\right)T}\times\nonumber\\
    &&\left(\|u_{0}\|^{2}_{\ell^{2}_{s}(\hbar\mathbb{Z}^{n})}+\|f_{\varepsilon}\|^{2}_{L^{2}([0,T];\ell^{2}_{s}(\hbar\mathbb{Z}^{n}))}\right)\nonumber\\
    &\leq&\tilde{a}_{0}^{-1}e^{K_{T}\left(\omega(\varepsilon)^{-L_{1}-1}+\omega(\varepsilon)^{-L_{1}-L_{2}}+\omega(\varepsilon)^{-L_{1}}\right)T}\times\nonumber\\
    &&\left(\|u_{0}\|^{2}_{\ell^{2}_{s}(\hbar\mathbb{Z}^{n})}+\|f_{\varepsilon}\|^{2}_{L^{2}([0,T];\ell^{2}_{s}(\hbar\mathbb{Z}^{n}))}\right),
\end{eqnarray*}
where $K_{T}=\tilde{a}_{0}^{-1}c_{1}T\max\{1,2c_{2},T^{-1}\}$.
 Putting $\omega(\varepsilon)\sim|\log(\varepsilon)|^{-1}$, we obtain
\begin{equation}\label{unorm}
	\|u_{\varepsilon}(t,\cdot)\|_{\ell^{2}_{s}(\hbar\mathbb{Z}^{n})}^{2}\lesssim \varepsilon^{-3L_{1}-L_{2}-1}\left(\|u_{0}\|_{\ell^{2}_{s}(\hbar\mathbb{Z}^{n})}^{2}+\|f_{\varepsilon}\|_{L^{2}([0,T];\ell^{2}_{s}(\hbar\mathbb{Z}^{n}))}^{2}\right),
\end{equation}
for all $t\in[0,T]$ and $\varepsilon\in(0,1]$. 

Since $(f_{\varepsilon})_{\varepsilon}$ is $L^{2}([0,T];\ell^{2}_{s}(\hbar\mathbb{Z}^{n}))$-moderate regularisation of $f$, there exist positive constants $L_{3}$ and $c$ such that
\begin{equation}\label{fmoderate}
\|f_{\varepsilon}\|_{L^{2}([0,T];\ell^{2}_{s}(\hbar\mathbb{Z}^{n}))}\leq c\varepsilon^{-L_{3}}.
\end{equation}
Now by integrating the estimate \eqref{unorm} with respect to $t\in[0,T]$ and then using \eqref{fmoderate}, we obtain
\begin{equation}
	\|u_{\varepsilon}\|_{L^{2}([0,T];\ell^{2}_{s}(\hbar\mathbb{Z}^{n}))}\lesssim \varepsilon^{-3L_{1}-L_{2}-L_{3}-1}.
\end{equation}
 Thus we conclude that $(u_{\varepsilon})_{\varepsilon}$ is $L^{2}([0,T];\ell^{2}_{s}(\hbar\mathbb{Z}^{n}))$-moderate. This completes the proof.
 \end{proof}
Therefore, we have established the existence of very weak solution for the Cauchy problem \eqref{heatpde} with irregular coefficients. We will now prove the uniqueness of the very weak solution of it in the sense of Definition \ref{uniquedef}.
\begin{proof}[Proof of Theorem \ref{uniq}]
	Let $(u_{\varepsilon})_{\varepsilon}$	 and $(\tilde{u}_{\varepsilon})_{\varepsilon}$	 be the families of solution corresponding to the Cauchy problems \eqref{reg1} and \eqref{reg2}, respectively. 
	Denoting $w_{\varepsilon}:=u_{\varepsilon}-\tilde{u}_{\varepsilon}$, we get
	\begin{equation}\label{weqn}
		\left\{\begin{array}{l}
			\partial_{t} w_{\varepsilon}(t, k)+a_{\varepsilon}(t)\hbar^{-2\alpha}\left(-\mathcal{L}_{\hbar}\right)^{\alpha}w_{\varepsilon}(t, k)+b_{\varepsilon}(t)w_{\varepsilon}(t, k)=g_{\varepsilon}(t, k),\quad t\in (0,T], \\
			w_{\varepsilon}(0, k)=0,\quad k \in \hbar\mathbb{Z}^{n}, 
		\end{array}\right.
	\end{equation}
 where 
	\begin{equation}
		g_{\varepsilon}(t,k):=(\tilde{a}_{\varepsilon}-a_{\varepsilon})(t) \hbar^{-2\alpha}\left(-\mathcal{L}_{\hbar}\right)^{\alpha}\tilde{u}_{\varepsilon}(t, k)+(\tilde{b}_{\varepsilon}-b_{\varepsilon})(t) \tilde{u}_{\varepsilon}(t, k)+(f_{\varepsilon}-\tilde{f}_{\varepsilon})(t,k).
	\end{equation}
	Since the nets $(\tilde{a}_{\varepsilon}-a_{\varepsilon})_{\varepsilon}$,  $(\tilde{b}_{\varepsilon}-b_{\varepsilon})_{\varepsilon}$ and $(f_{\varepsilon}-\tilde{f}_{\varepsilon})_{\varepsilon}$ are $L^{\infty}_{1}$-negligible, $L^{\infty}$-negligible and  $L^{2}([0,T];\ell^{2}_{s}(\hbar\mathbb{Z}^{n}))$-negligible, respectively, it follows that $g_{\varepsilon}$ is $L^{2}([0,T];\ell^{2}_{s}(\hbar\mathbb{Z}^{n}))$-negligible.
	
	Taking the Fourier transform of the Cauchy problem \eqref{weqn} with respect to $k\in\hbar\mathbb{Z}^{n}$, we obtain
	\begin{equation}\label{uniquecp}
		\left\{\begin{array}{l}
			\partial_{t} \widehat{w}_{\varepsilon}(t, \xi)+a_{\varepsilon}(t)\nu^{2}(\xi)\widehat{w}_{\varepsilon}(t, \xi)+b_{\varepsilon}(t)\widehat{w}_{\varepsilon}(t, \xi)=\widehat{g}_{\varepsilon}(t, \xi),\quad(t, \xi) \in(0, T] \times \mathbb{T}_{\hbar}^{n}, \\
			\widehat{w}_{\varepsilon}(0, \xi)=0,\quad \xi \in \mathbb{T}_{\hbar}^{n},
		\end{array}\right.
	\end{equation}
where $\nu^{2}(\xi)$ is given by \eqref{nuvalue}. The energy functional for the Cauchy problem \eqref{uniquecp} is given by
\begin{equation}\label{uniqueenergy}
	E_{\varepsilon}(t,\xi):=(a_{\varepsilon}(t)\widehat{w}_{\varepsilon}(t,\xi),\widehat{w}_{\varepsilon}(t,\xi)),\quad (t,\xi) \in[0, T]\times\mathbb{T}_{\hbar}^{n},
\end{equation}
where $(\cdot,\cdot)$ is the inner product in the Sobolev space $H^{s}(\mathbb{T}_{\hbar}^{n})$.
Then using \eqref{std}, we have the following energy bounds
\begin{equation}\label{uniqueenergyest}
	\tilde{a}_{0}\|\widehat{w}_{\varepsilon}(t,\cdot)\|^{2}_{H^{s}}\leq E_{\varepsilon}(t,\xi)\leq    c_{1} \omega(\varepsilon)^{-L_{1}}\|\widehat{w}_{\varepsilon}(t,\cdot)\|^{2}_{H^{s}},\quad (t,\xi) \in[0, T]\times\mathbb{T}_{\hbar}^{n}.
\end{equation}
Further we can calculate
\begin{equation}\label{uniquueengder}
	\partial_{t}E_{\varepsilon}(t,\xi)
	\leq\left(|a^{\prime}_{\varepsilon}(t)|+2|a_{\varepsilon}(t)||b_{\varepsilon}(t)|+|a_{\varepsilon}(t)|\right) \|\widehat{w}_{\varepsilon}(t,\cdot)\|^{2}_{H^{s}}
	+|a_{\varepsilon}(t)|\|\widehat{g}_{\varepsilon}(t,\cdot)\|^{2}_{H^{s}}.
\end{equation}
 
Combining the estimates \eqref{std}, \eqref{uniqueenergyest} and \eqref{uniquueengder}, we obtain
\begin{equation}\label{uniquegroneqn}
	\partial_{t}E_{\varepsilon}(t,\xi)
	\leq\kappa_{1}\left(\omega(\varepsilon)^{-L_{1}-1}+\omega(\varepsilon)^{-L_{1}-L_{2}}+\omega(\varepsilon)^{-L_{1}}\right)E_{\varepsilon}(t,\xi)+c_{1} \omega(\varepsilon)^{-L_{1}}\|\widehat{g}_{\varepsilon}(t,\cdot)\|^{2}_{H^{s}},
\end{equation}
where $\kappa_{1}=\tilde{a}_{0}^{-1}c_{1}\max\{1,2c_{2}\}$. Applying the Gronwall’s lemma
to the inequality \eqref{uniquegroneqn}, we obtain
\begin{multline}\label{uniquegron}
	E_{\varepsilon}(t,\xi)\leq\\ e^{\int_{0}^{t}\kappa_{1}\left(\omega(\varepsilon)^{-L_{1}-1}+\omega(\varepsilon)^{-L_{1}-L_{2}}+\omega(\varepsilon)^{-L_{1}}\right)\mathrm{d}\tau}\left(E_{\varepsilon}(0,\xi)+c_{1} \omega(\varepsilon)^{-L_{1}}\int_{0}^{t}\|\widehat{g}_{\varepsilon}(\tau,\cdot)\|^{2}_{H^{s}}\mathrm{d}\tau\right).
\end{multline}
Putting together \eqref{uniqueenergyest} and \eqref{uniquegron}, and then using the fact that $\widehat{w}_{\varepsilon}(0,\xi)\equiv0$ for all $\varepsilon\in(0,1]$, we get
\begin{eqnarray*}
	\|\widehat{w}_{\varepsilon}(t,\cdot)\|^{2}&\leq&\tilde{a}_{0}^{-1}c_{1}\omega(\varepsilon)^{-L_{1}}e^{\kappa_{1}\left(\omega(\varepsilon)^{-L_{1}-1}+\omega(\varepsilon)^{-L_{1}-L_{2}}+\omega(\varepsilon)^{-L_{1}}\right)T}\int_{0}^{T}\|\widehat{g}_{\varepsilon}(\tau,\cdot)\|^{2}_{H^{s}}\mathrm{d}\tau\nonumber\\
	&\leq&\tilde{a}_{0}^{-1}c_{1}e^{\kappa_{T}\left(\omega(\varepsilon)^{-L_{1}-1}+\omega(\varepsilon)^{-L_{1}-L_{2}}+\omega(\varepsilon)^{-L_{1}}\right)}\int_{0}^{T}\|\widehat{g}_{\varepsilon}(\tau,\cdot)\|^{2}_{H^{s}}\mathrm{d}\tau,
\end{eqnarray*}
where $\kappa_{T}=c_{1}T\max\{1,2c_{2},T^{-1}\}$. Putting $\omega(\varepsilon)\sim|\log(\varepsilon)|^{-1}$ and using the Plancherel formula \eqref{planch}, we obtain
\begin{equation*}\label{wnorm}
	\|w_{\varepsilon}(t,\cdot)\|_{\ell^{2}_{s}(\hbar\mathbb{Z}^{n})}^{2}\lesssim \varepsilon^{-3L_{1}-L_{2}-1}\|g_{\varepsilon}\|_{L^{2}([0,T];\ell^{2}_{s}(\hbar\mathbb{Z}^{n}))}^{2},
\end{equation*}
for all $t\in[0,T]$. Since $g_{\varepsilon}$ is $L^{2}\left([0, T] ; \ell^{2}_{s}(\hbar\mathbb{Z}^{n})\right)$-negligible, we obtain
\begin{equation*}
\left\|w_{\varepsilon}(t,\cdot)\right\|_{\ell^{2}_{s}(\hbar\mathbb{Z}^{n})}^{2} \lesssim \varepsilon^{-3 L_1-L_2-1} \varepsilon^{3 L_1+L_2+1+q}=\varepsilon^{q}, \quad\text{ for all } q \in \mathbb{N}_0,
\end{equation*}
for all $t\in[0,T]$. Now by integrating the above estimate  with respect to $t\in[0,T]$, we get 
\begin{equation*}
\left\|w_{\varepsilon}\right\|_{L^{2}([0,T];\ell^{2}_{s}(\hbar\mathbb{Z}^{n}))} \lesssim \varepsilon^{q}, \quad\text{ for all } q \in \mathbb{N}_0.
\end{equation*}
  Thus $\left(u_{\varepsilon}-\tilde{u}_{\varepsilon}\right)_{\varepsilon}$ is $L^{2}\left([0, T] ; \ell^{2}_{s}(\hbar\mathbb{Z}^{n})\right)$-negligible. This completes the proof.
\end{proof}
 Next we prove that the
 very weak solution obtained in Theorem \ref{ext} is consistent with the classical solution obtained in Theorem \ref{heatwellpo}.
\begin{proof}[Proof of Theorem \ref{cnst}]
	Let $\tilde{u}$ be the classical solution given by Theorem \ref{heatwellpo}, that is,  $\tilde{u}$ satisfies the Cauchy problem 
	\begin{equation}\label{cnst1}
		\left\{\begin{array}{l}
			\partial_{t} \tilde{u}(t, k)+a(t) \hbar^{-2\alpha}\left(-\mathcal{L}_{\hbar}\right)^{\alpha}\tilde{u}(t, k)+b(t) \tilde{u}(t, k)=f(t,k),\quad(t, k) \in(0, T] \times \hbar\mathbb{Z}^{n}, \\
			\tilde{u}(0, k)=u_{0}(k),\quad k \in \hbar\mathbb{Z}^{n}, 
		\end{array}\right.
	\end{equation} 	
	and let  $(u_{\varepsilon})_{\varepsilon}$  be the very weak solution obtained by Theorem \ref{ext}, that is, $(u_{\varepsilon})_{\varepsilon}$ satisfies the regularised Cauchy problem 
	\begin{equation}\label{cnst2}
		\left\{\begin{array}{l}
			\partial_{t} u_{\varepsilon}(t, k)+a_{\varepsilon}(t) \hbar^{-2\alpha}\left(-\mathcal{L}_{\hbar}\right)^{\alpha}u_{\varepsilon}(t, k)+b_{\varepsilon}(t) u_{\varepsilon}(t, k)=f_{\varepsilon}(t,k),\quad t \in(0, T], \\
			u_{\varepsilon}(0, k)=u_{0}(k),\quad k \in \hbar\mathbb{Z}^{n}.
		\end{array}\right.
	\end{equation} 
	Note that by the hypothesis the nets $\left(a_{\varepsilon}-a\right)_{\varepsilon},\left(b_{\varepsilon}-b\right)_{\varepsilon}$ and $\left(f_{\varepsilon}-f\right)_{\varepsilon}$ are converging to  $0$ uniformly.  Using \eqref{cnst1} and \eqref{cnst2}, we have
	\begin{equation}\label{cnst3}
		\left\{\begin{array}{l}
			\partial_{t} \tilde{u}(t, k)+a_{\varepsilon}(t)\hbar^{-2\alpha}\left(-\mathcal{L}_{\hbar}\right)^{\alpha}\tilde{u}(t, k)+b_{\varepsilon}(t) \tilde{u}(t, k)=f_{\varepsilon}(t,k)+g_{\varepsilon}(t,k),\quad t \in(0, T], \\
			\tilde{u}(0, k)=u_{0}(k),\quad k \in \hbar\mathbb{Z}^{n},
		\end{array}\right.
	\end{equation} 	
	where
	\begin{equation*}
		g_{\varepsilon}(t,k):=\left(a_{\varepsilon}-a\right)(t) \hbar^{-2\alpha}\left(-\mathcal{L}_{\hbar}\right)^{\alpha}\tilde{u}(t, k)+\left(b_{\varepsilon}-b\right)(t) \tilde{u}(t, k)+\left(f-f_{\varepsilon}\right)(t,k),
	\end{equation*}
	$g_{\varepsilon} \in L^{2}([0, T] ; \ell^{2}_{s}(\hbar\mathbb{Z}^{n}))$ and  $g_{\varepsilon}\to0$ in $L^{2}([0, T] ; \ell^{2}_{s}(\hbar\mathbb{Z}^{n}))$ as $\varepsilon \rightarrow 0$. 

 Combining the Cauchy problem
	 \eqref{cnst2} and \eqref{cnst3}, we deduce  that the net $w_{\varepsilon}:=\left(\tilde{u}-u_{\varepsilon}\right)$  solves the Cauchy problem
	\begin{equation}\label{cnst4}
		\left\{\begin{array}{l}
			\partial_{t} w_{\varepsilon}(t, k)+a_{\varepsilon}(t) \hbar^{-2\alpha}\left(-\mathcal{L}_{\hbar}\right)^{\alpha}w_{\varepsilon}(t, k)+b_{\varepsilon}(t) w_{\varepsilon}(t, k)=g_{\varepsilon}(t,k),\quad t \in(0, T], \\
			w_{\varepsilon}(0, k)=0,\quad k \in \hbar\mathbb{Z}^{n}.
		\end{array}\right.
	\end{equation}
 Taking the Fourier transform of the Cauchy problem \eqref{cnst4} with respect to $k\in\hbar\mathbb{Z}^{n}$, we obtain
	\begin{equation}\label{consistcp}
	\left\{\begin{array}{l}
		\partial_{t} \widehat{w}_{\varepsilon}(t, \xi)+a_{\varepsilon}(t)\nu^{2}(\xi)\widehat{w}_{\varepsilon}(t, \xi)+b_{\varepsilon}(t)\widehat{w}_{\varepsilon}(t, \xi)=\widehat{g}_{\varepsilon}(t, \xi),\quad(t, \xi) \in(0, T] \times \mathbb{T}_{\hbar}^{n}, \\
		\widehat{w}_{\varepsilon}(0, \xi)=0,\quad \xi \in \mathbb{T}_{\hbar}^{n},
	\end{array}\right.
\end{equation}
where $\nu^{2}(\xi)$ is given by \eqref{nuvalue}. Define
the energy functional for the Cauchy problem \eqref{consistcp} by
\begin{equation*}
	E_{\varepsilon}(t,\xi):=( a_{\varepsilon}(t)\widehat{w}_{\varepsilon}(t,\xi),\widehat{w}_{\varepsilon}(t,\xi)),\quad (t,\xi) \in[0, T]\times\mathbb{T}_{\hbar}^{n},
\end{equation*}
where $(\cdot,\cdot)$ is the inner product in the Sobolev space $H^{s}(\mathbb{T}_{\hbar}^{n})$. 

Since the coefficients are sufficiently regular, following  the lines of Theorem \ref{heatwellpo}, the next energy estimate holds true
	\begin{equation*}
		\partial_t E_{\varepsilon}(t, \xi) \leq \kappa_{1}E_{\varepsilon}(t, \xi)+\kappa_{2}\left|\widehat{g}_{\varepsilon}(t, \xi)\right|^{2},
	\end{equation*}
for some positive constants $\kappa_{1}$ and $\kappa_{2}$. Then using the Gronwall's lemma and the  energy bounds similar to \eqref{engyest} alongwith the Plancherel formula \eqref{planch}, we obtain the following estimate
	\begin{equation*}
		\|w_{\varepsilon}(t,\cdot)\|^{2}_{\ell^{2}_{s}(\hbar\mathbb{Z}^{n})}\lesssim \|w_{\varepsilon}(0,\cdot)\|_{\ell^{2}_{s}(\hbar\mathbb{Z}^{n})}^{2}+\|g_{\varepsilon}\|_{L^{2}([0,T];\ell^{2}_{s}(\hbar\mathbb{Z}^{n}))}^{2},
	\end{equation*}
+	for all $t\in[0,T]$. Now by integrating the above estimate with respect to $t\in[0,T]$ and using the fact that $w_{\varepsilon}(0, k)\equiv0$ for all $\varepsilon\in(0,1]$, we get
 \begin{equation*}
		\|w_{\varepsilon}\|_{L^{2}([0,T];\ell^{2}_{s}(\hbar\mathbb{Z}^{n}))}^{2}\lesssim \|g_{\varepsilon}\|_{L^{2}([0,T];\ell^{2}_{s}(\hbar\mathbb{Z}^{n}))}^{2}.
	\end{equation*}
	Since $g_{\varepsilon} \rightarrow 0$ in $L^{2}([0, T] ; \ell^{2}_{s}(\hbar\mathbb{Z}^{n}))$, we have
	\begin{equation*}
		w_{\varepsilon} \to 0 \text{ in }L^{2}([0, T] ; \ell^{2}_{s}(\hbar\mathbb{Z}^{n})),\quad \varepsilon\to 0,
	\end{equation*}
	implying also that
	\begin{equation*}
		u_{\varepsilon} \to \tilde{u} \text{ in }L^{2}([0, T] ; \ell^{2}_{s}(\hbar\mathbb{Z}^{n})) ,\quad \varepsilon\to 0.
	\end{equation*} 
	Furthermore, the limit is the same for every representation of $u$, since they will differ from $\left(u_{\varepsilon}\right)_{\varepsilon}$ by a $L^{2}([0, T];\ell^{2}_{s}(\hbar\mathbb{Z}^{n}))$-negligible net. This completes the proof.
\end{proof}
\section{Semi-classical limit $\hbar\to 0$}\label{sec:limit}
In this section, we will prove the semi-classical limit theorems for the classical solution as well as for the very weak solution.
\begin{proof}[Proof of Theorem \ref{semlmtclass}]
Consider two Cauchy problems:
\begin{equation}\label{CP1}
	\left\{\begin{array}{l}
		\partial_{t} u(t, k)+a(t)\hbar^{-2\alpha}(-\mathcal{L}_{\hbar})^{\alpha} u(t, k)+b(t) u(t, k)=f(t,k), \quad  (t,k) \in(0,T]\times\hbar\mathbb{Z}^{n}, \\
		u(0, k)=u_{0}(k), \quad k\in\hbar\mathbb{Z}^{n},
	\end{array}\right.
\end{equation}
and
\begin{equation}\label{CP2}
	\left\{
	\begin{array}{ll}
		\partial_{t}v(t,x)+a(t)(-\mathcal{L})^{\alpha}v(t,x)+b(t)v(t,x)=f(t,x), \quad (t,x) \in(0,T]\times\mathbb{R}^{n},\\
		v(0,x)=u_{0}(x),\quad x\in\mathbb{R}^{n},
	\end{array}
	\right.
\end{equation}
where $\left(-\mathcal{L}\right)^{\alpha}$ is the usual fractional Laplacian on $\mathbb{R}^{n}$ given by \eqref{frlap}. We have assumed that $f$ and $u_0$ in \eqref{CP1} are the restrictions in $\hbar \mathbb{Z}^n$ of the corresponding ones in \eqref{CP2} defined on $\mathbb{R}^n$. From the equations \eqref{CP1} and \eqref{CP2}, 
denoting $w:=u-v$,  we get
\begin{equation}\label{CPF}
	\left\{
	\begin{array}{ll}
		\partial_{t}w(t,k)+a(t)\hbar^{-2\alpha}(-\mathcal{L}_{\hbar})^{\alpha}w(t,k)+b(t)w(t,k)=a(t)\left((-\mathcal{L})^{\alpha}-\hbar^{-2\alpha}(-\mathcal{L}_{\hbar})^{\alpha}\right)v(t,k),\\
		w(0,k)=0,\quad k\in\hbar\mathbb{Z}^{n}.               
	\end{array}
	\right.
\end{equation}
Since $w_{0} = 0$, applying Theorem \ref{heatwellpo} with s = 0 for the Cauchy problem \eqref{CPF} and 
using the estimate \eqref{wellpoest}, we obtain
\begin{eqnarray}\label{WP111}
	\left\|w(t,\cdot)\right\|^{2}_{\ell^{2}(\hbar\mathbb{Z}^{n})} 
	&\leq& C_{T,a,b}\left\|a\left((-\mathcal{L})^{\alpha}-\hbar^{-2\alpha}(-\mathcal{L}_{\hbar})^{\alpha}\right)v\right\|^{2}_{L^{2}([0,T];\ell^{2}(\hbar\mathbb{Z}^{n}))}\nonumber\\
 &\leq&C_{T,a,b}\left\|a\right\|^{2}_{L^{\infty}([0,T])}\left\|\left((-\mathcal{L})^{\alpha}-\hbar^{-2\alpha}(-\mathcal{L}_{\hbar})^{\alpha}\right)v\right\|^{2}_{L^{2}([0,T];\ell^{2}(\hbar\mathbb{Z}^{n}))},\nonumber\\
\end{eqnarray}
for all $t\in[0,T]$, where the constant $C_{T,a,b}$ is given by
\begin{equation*}
  C_{T,a,b}=a_{0}^{-1}\|a\|_{L^{\infty}}e^{a_{0}^{-1}(\|a_{t}\|_{L^{\infty}}+2\|a\|_{L^{\infty}}\|b\|_{L^{\infty}}+\|a\|_{L^{\infty}})T}.
\end{equation*}

 Now we will estimate the term $\left\|\left((-\mathcal{L})^{\alpha}-\hbar^{-2\alpha}(-\mathcal{L}_{\hbar})^{\alpha}\right)v\right\|^{2}_{L^{2}([0,T];\ell^{2}(\hbar\mathbb{Z}^{n}))}.$ 
Using the Plancherel formula \eqref{planch} for $s=0$, \eqref{lapft} and \eqref{lapeucft}, we have
\begin{multline}\label{fnm}
	\left\|\left((-\mathcal{L})^{\alpha}-\hbar^{-2\alpha}(-\mathcal{L}_{\hbar})^{\alpha}\right)v(t,\cdot)\right\|_{\ell^{2}(\hbar\mathbb{Z}^{n})}=\\ \left\|\left(\left[\sum_{l=1}^n (2\pi\xi_{l})^{2}\right]^{\alpha}-\hbar^{-2\alpha}\left[\sum_{l=1}^n 4 \sin ^2\left(\pi\hbar \xi_{l}\right)\right]^{\alpha}\right)\widehat{v}(t,\cdot)\right\|_{L^{2}(\mathbb{T}_{\hbar}^{n})}.
\end{multline}
Since  $|\cdot|^{\alpha}:\mathbb{R}\to \mathbb{R}$ is $\alpha$-H\"older continuous for $0<\alpha\leq1$,  we have the following inequality
\begin{equation*}
	\left||x|^{2\alpha}-|y|^{2\alpha}\right|\lesssim \left||x|^{2}-|y|^{2}\right|^{\alpha}, \quad x,y\in\mathbb{R}^{n},
\end{equation*} 
 and  $|x|=\sqrt{x_{1}^{2}+\dots+x_{n}^{2}}$. Now, using the above inequality and the Taylor expansion for $\sin^{2}(\pi\hbar\xi_{l})$, we get
\begin{eqnarray}\label{diffnorm}
	&&\left|\left[\sum_{l=1}^n (2\pi\xi_{l})^{2}\right]^{\alpha}-\hbar^{-2\alpha}\left[\sum_{l=1}^n 4 \sin ^2\left(\pi\hbar \xi_{l}\right)\right]^{\alpha}\right|\nonumber\\
	&\lesssim&
	\left|\sum\limits_{l=1}^{n}4\pi^2\xi^{2}_{l}-\hbar^{-2}\sum\limits_{l=1}^{n}4\sin^{2}(\pi\hbar\xi_{l})\right|^{\alpha}\nonumber\\
	&=&\left|\sum\limits_{l=1}^{n}\left[4\pi^2\xi^{2}_{l}-\hbar^{-2}4\left(\pi^{2}\hbar^{2}\xi_{l}^{2}-\frac{\pi^{4}}{3}\hbar^{4}\xi_{l}^{4}\cos\left(2\theta_{l}\right)\right)\right]\right|^{\alpha}\nonumber\\
	&=&\left(\frac{4\pi^{4}\hbar^{2}}{3}\right)^{\alpha}\left|\sum\limits_{l=1}^{n}\xi_{l}^{4}\cos(2\theta_{l})\right|^{\alpha}\nonumber\\
	&\lesssim&\hbar^{2\alpha}\left[\sum\limits_{l=1}^{n}\xi^{4}_{l}\right]^{\alpha}\nonumber\\
	&\lesssim&\hbar^{2\alpha}|\xi|^{4\alpha},
\end{eqnarray}
where $|\xi|^{4\alpha}=\left[\sum\limits_{l=1}^{n}\xi^{2}_{l}\right]^{2\alpha}$ and $\theta_{l}\in(0,\pi\hbar\xi_{l})$ or $(\pi\hbar\xi_{l},0)$ depending on the sign of $\xi_{l}$.  Now, combining the  estimate \eqref{diffnorm} with  \eqref{fnm}, we get
\begin{eqnarray}\label{fnmm}
\left\|\left((-\mathcal{L})^{\alpha}-\hbar^{-2\alpha}(-\mathcal{L}_{\hbar})^{\alpha}\right)v(t,\cdot)\right\|_{\ell^{2}(\hbar\mathbb{Z}^{n})}&\lesssim& \hbar^{2\alpha}\|(1+|\cdot|^{2})^{2\alpha}\widehat{v}(t,\cdot)\|_{L^{2}(\mathbb{T}_{\hbar}^{n})} \nonumber\\
&\lesssim& \hbar^{2\alpha}\|(1+|\cdot|^{2})^{m/2}\widehat{v}(t,\cdot)\|_{L^{2}(\mathbb{T}_{\hbar}^{n})},
\end{eqnarray}
whenever $m\geq 4\alpha$. Since $u_{0}\in H^{m}(\mathbb{R}^{n})$ with $m\geq 4\alpha$, using Theorem \ref{Eucheatwellpo}, we have $v\in H^{m}(\mathbb{R}^{n})$  with $m\geq 4\alpha$.
Therefore, using \eqref{WP111} and \eqref{fnmm},  we deduce that $\left\|w(t,\cdot)\right\|_{\ell^{2}(\hbar\mathbb{Z}^{n})}\to 0$ as $\hbar\to 0.$  Hence we have 
\begin{equation*}
    \left\|v(t,\cdot)-u(t,\cdot)\right\|_{\ell^{2}(\hbar\mathbb{Z}^{n})}\to 0\text{ as } \hbar\to 0,\text{ for all } t\in[0,T].
\end{equation*}
 This finishes the proof of Theorem \ref{semlmtclass}.
\end{proof}
\begin{proof}[Proof of Theorem \ref{semlmtvwk}]
    Without making any significant changes to the proof of Theorem \ref{semlmtclass}, we can prove Theorem \ref{semlmtvwk}.
\end{proof}
\section{Remarks}\label{sec:remark}
In this section we make a few remarks related to the very weak solution for the Cauchy problem \eqref{heatpdeEuc} in the Euclidean setting:
\begin{enumerate}
    \item 
    For a net $u_\varepsilon=u_\varepsilon(t,x)$, where $x \in \mathbb{R}^n$ is a variable in the Euclidean space, the definitions of moderateness and negligibility are adapted accordingly from Definition \eqref{cinfm}; i.e.,  the net $(u_\varepsilon)_\varepsilon \in {L^{2}([0,T];H^{m}(\mathbb{R}^{n}))}$  is $L^{2}([0,T];H^{m}(\mathbb{R}^{n}))$-moderate if  there exist $N \in \mathbb{N}_{0}$ and $c>0$ such that
    \begin{equation*}
       \|u_{\varepsilon}\|_{L^{2}([0,T];H^{m}(\mathbb{R}^{n}))} \leq c \varepsilon^{-N},     
    \end{equation*}
    for all $\varepsilon\in(0,1]$ and is $L^{2}([0,T];H^{m}(\mathbb{R}^{n}))$-negligible if for all $q\in\mathbb{N}_{0}$ there exists $c>0$  such that
    \begin{equation*}
    \|u_{\varepsilon}\|_{L^{2}([0,T];H^{m}(\mathbb{R}^{n}))} \leq c \varepsilon^{q},
	\end{equation*}
   for all $\varepsilon\in(0,1]$.
\item The notion of a very weak solution for the Cauchy problem \eqref{heatpdeEuc} can be adapted from Definition \ref{vwkdef} by simply replacing the $L^{2}([0, T] ; \ell^{2}_{s}(\hbar\mathbb{Z}^{n}))$-moderate regularisation by the $L^{2}([0,T];H^{m}(\mathbb{R}^{n}))$-moderate regularisation for the source term $f$ and the net $(u_{\varepsilon})_{\varepsilon}.$
\item The proof of Theorem \ref{Eucheatwelvws} will be similar to the proof of Theorem \ref{ext} except for using the inner product of $L^{2}_{m}(\mathbb{R}^{n})$ instead of $H^{s}(\mathbb{T}_{\hbar}^{n})$ in \eqref{weakenergy}.
\item The notion of the uniqueness of very weak solution for the Cauchy problem \eqref{heatpdeEuc} can also be formulated by making similar modifications to Definition  \ref{uniquedef}.
\item If the coefficients $a,b$ are regular, the very weak solution obtained in Theorem \ref{Eucheatwelvws} recaptures the  classical solution obtained in Theorem  \ref{Eucheatwellpo}   in the limit $L^{2}([0,T];H^{m}(\mathbb{R}^{n}))$ as $\varepsilon\to 0$. More precisely, we have the consistency result similar to Theorem \ref{cnst} with the same modifications as above.
\end{enumerate}
\bibliographystyle{alphaabbr}
\bibliography{discreteFractionalLaplace}

\end{document}